\documentclass[a4paper]{amsart}
\usepackage[latin1]{inputenc}
\usepackage{amssymb}
\usepackage{amsmath}
\usepackage{mathrsfs}
\usepackage{amsthm}
\usepackage{amsfonts}
\usepackage{textcomp}
\usepackage{graphicx}
\usepackage[pdftex]{color}
\usepackage{paralist}
\usepackage[shortlabels]{enumitem}
\usepackage{hyperref}
\usepackage{comment}
\usepackage[arrow, matrix, curve]{xy}
\usepackage{tikz-cd}
\usepackage[square,sort,comma,numbers]{natbib} 

\newtheorem*{corollary*}{Corollary}
\newtheorem*{conjecture*}{Conjecture}
\newtheorem*{example*}{Example}
\newtheorem*{theorem*}{Theorem}
\newtheorem*{proposition*}{Proposition}

\newtheorem{theorem}{Theorem}[section]

\newtheorem{corollary}[theorem]{Corollary}
\newtheorem{lemma}[theorem]{Lemma}
\newtheorem{proposition}[theorem]{Proposition}

\newtheorem*{claim*}{Claim}
\newtheorem{conjecture}[theorem]{Conjecture}

\newtheorem{thmintroduction}{Theorem}

\theoremstyle{definition}

\newtheorem{problem}[theorem]{Problem}
\newtheorem{remark}[theorem]{Remark}
\newtheorem{example}[theorem]{Example}

\theoremstyle{remark}

\numberwithin{equation}{section}

\theoremstyle{plain}

\newtheoremstyle{cited}%
  {4pt}
  {4pt}
  {\itshape}
  {}
  {\bfseries}
  {.}
  {0.5em}
  {\thmname{#1}  \thmnote{\normalfont#3}}

\theoremstyle{cited}

\newtheorem*{corintroduction}{Corollary}

\makeatletter
\renewcommand*\env@matrix[1][\
arraystretch]{%
  \edef\arraystretch{#1}%
  \hskip -\arraycolsep
  \let\@ifnextchar\new@ifnextchar
  \array{*\c@MaxMatrixCols c}}
\makeatother

\newcommand{\Thmref}{Theorem~}

\newcommand{\proj}{\operatorname{proj}}
\newcommand{\inj}{\operatorname{inj}}

\newcommand{\Ext}{\operatorname{Ext}}

\newcommand{\domdim}{\operatorname{domdim}}

\newcommand{\gldim}{\operatorname{gldim}}
\newcommand{\End}{\operatorname{End}}

\newcommand{\pdim}{\operatorname{pdim}}
\newcommand{\idim}{\operatorname{idim}}

\newcommand{\Tr}{\operatorname{tr}}
\newcommand{\transpose}{\operatorname{Tr}}
\newcommand{\Hom}{\operatorname{Hom}}
\newcommand{\im}{\operatorname{im}}
\newcommand{\codomdim}{\operatorname{codomdim}}
\newcommand{\add}{\operatorname{\mathrm{add}}}
\renewcommand{\top}{\operatorname{\mathrm{top}}}
\newcommand{\rad}{\operatorname{\mathrm{rad}}}
\newcommand{\soc}{\operatorname{\mathrm{soc}}}

\newcommand{\depth}{\operatorname{depth}}
\newcommand{\rmod}{\operatorname{mod-}\!\!}
\newcommand{\cogen}{\operatorname{Cogen}}
\newcommand{\gen}{\operatorname{Gen}}

\begin{document}

\author[T. Cruz]{Tiago Cruz}
\address[Tiago Cruz]{Institut f\"ur Algebra und Zahlentheorie,
Universit\"at Stuttgart, Germany }
\email{tiago.cruz@mathematik.uni-stuttgart.de}

\author[R. Marczinzik]{Ren\'{e} Marczinzik}
\address[Ren\'{e} Marczinzik]{Mathematical Institute of the University of Bonn, Endenicher Allee 60, 53115 Bonn, Germany}
\email{marczire@math.uni-bonn.de}
\title{An Auslander-Buchsbaum formula for higher
Auslander algebras and applications}
\date{\today}

\begin{abstract}
We provide a new non-commutative generalisation of the Auslander-Buchsbaum formula for higher Auslander algebras and use this to show that the class of tilted Auslander algebras, studied recently by Zito, and QF-1 algebras of global dimension at most 2, studied by Ringel in the 1970s, coincide. We furthermore give an explicit classification of this class of algebras and present generalisations to higher homological dimensions with a new local characterisation of QF-1 higher Auslander algebras. 
\end{abstract}


\maketitle

\vspace{-0.5cm}

\section{Introduction}
One of the central results in commutative algebra is the Auslander-Buchsbaum formula, which states that 
\begin{equation} 
\depth M+ \pdim M= \depth R
\end{equation}
for commutative local noetherian rings $R$ and modules $M$ of finite projective dimension.
For regular commutative local rings, this formula can be generalised  as follows using their Calabi-Yau properties:
\begin{theorem} \label{CYregularrings}
Let $R$ be a commutative local noetherian ring with maximal ideal $\mathfrak{m}$ and $K:=R/\mathfrak{m}$ and let $M$ be an $R$-module.
\begin{enumerate}
    \item $\depth M= \inf \{ n \geq 0 \mid \Ext_R^n(K,M) \neq 0\}.$
    \item $\pdim M= \sup \{ n \geq 0 \mid \Ext_R^n(M,K) \neq 0 \}$.
    \item If $R$ is additionally regular, that is $d=\gldim R= \depth R$, then
    $\Ext_R^n(K,M) \cong \Ext_R^{d-n}(M,K).$
\end{enumerate}
\end{theorem}
Part (1) and (2) are well known, while part (3) can be found in \cite[Proposition 3.10]{IR}.
Part (3) immediately implies the Auslander-Buchsbaum formula together with the formulas for the projective dimension and the depth from (1) and (2).

Considerable interest has been shown along the years to obtain non-commutative generalisations of the Auslander-Buchsbaum formula. For instance, it was generalised to the non-commutative graded case in \cite{MR1644217}, to the non-commutative local case in \cite{MR968204}, \cite{MR1741569}  and for subclasses of AS-Gorenstein algebras in \cite{MR1848957}.

The aim of this article is to give a non-commutative analogue of Theorem \ref{CYregularrings} for finite-dimensional algebras over a field and give applications, where the role of the depth is replaced with the dominant dimension. 
Let $A$ now be a finite-dimensional algebra over a field $K$. Without loss of generality, we assume that $A$ is connected and not semisimple. 
Recall that the \emph{dominant dimension}, $\domdim M$, of a finitely generated module $M$ over a finite-dimensional algebra $A$ is defined as the smallest $n$ such that $I^n$ is non-projective or infinite when no such $n$ exists, where 
$$0 \rightarrow M \rightarrow I^0 \rightarrow I^1 \rightarrow I^2 \rightarrow \cdots $$ is a minimal injective coresolution of $M$. The dominant dimension of an algebra is then simply defined as the dominant dimension of the regular module. A local commutative noetherian ring $R$ is \emph{regular} precisely when its global dimension, $\gldim R$, is finite, or equivalently, if $\gldim R= \depth R$ holds. For finite-dimensional algebras, this condition is then replaced by $\gldim A = \domdim A<+\infty$, which are exactly the (non-semisimple) \emph{higher Auslander algebras}. Higher Auslander algebras were introduced by Iyama in \cite{Iya} and are the foundations of higher Auslander-Reiten theory. One of the central results is the higher Auslander correspondence that shows that higher Auslander algebras are in a bijective correspondence to cluster tilting modules. Higher Auslander-Reiten theory and in particular cluster tilting modules played a key role in the recent development of the representation theory of finite-dimensional algebras with many interactions to other areas such as algebraic geometry, combinatorics and Lie theory, we refer for example to \cite{Iya3} and \cite{GLS} for surveys. Let $e$ be an idempotent such that $\add eA$ is the subcategory of projective-injective right $A$-modules and let $f$ be an idempotent such that $\add Af$ is the subcategory of projective-injective left $A$-modules. Then we define $\underline{A}:=A/AfA$ and $\overline{A}:=A/AeA$. Note that when $A$ has dominant dimension at least two, then by the Morita-Tachikawa correspondence we can write $A \cong \End_B(M)$ for some algebra $B$ with generator-cogenerator $M$ of $\rmod B$. Then for algebras $A$ of dominant dimension at least two, we have as algebras $\underline{A} \cong \underline{\End}_B(M)$ and $\overline{A} \cong \overline{\End}_B(M)$. Motivated by this, we call $\underline{A}$ the \emph{stable} $A$-module and $\overline{A}$ the \emph{costable} $A$-module.
Our first main result gives a non-commutative analogue of Theorem \ref{CYregularrings}:
\begin{thmintroduction}[see Theorems \ref{domdimprojdimformulas} and \ref{thm3dot7}] \label{theoremA}
Let $A$ be a finite-dimensional algebra and $M$ a right $A$-module.
\begin{enumerate}
\item $\domdim M:= \inf \{ n \geq 0 \mid \Ext_A^n(\underline{A},M) \neq 0 \},$ if $M$ is not projective-injective.
\item Assume that $M$ has finite projective dimension over $A$. Then, $$\pdim M= \sup \{ n \geq 0 \mid \Ext_A^n(M,\overline{A}) \neq 0 \}.$$ 
\item If $A$ is additionally higher Auslander, that is $\gldim A=d=\domdim A$ and $M$ is non-projective-injective, then
$\Ext_A^n(\underline{A},M) \cong D \Ext_A^{d-n}(M,\overline{A})$ for every integer $n$.
\end{enumerate}
\end{thmintroduction}

We give several applications of this theorem in this article.
The first application is the following non-commutative analogue of the Auslander-Buchsbaum formula for higher Auslander algebras:
\begin{corintroduction}[\ref{cor3dot8}]
Let $A$ be a higher Auslander algebra and $M$ a module that is not projective-injective. Then $\pdim M + \domdim M=\gldim A.$
\end{corintroduction}
This formula was also obtained in \cite{CIM} using different methods and relying on results of Miyachi \cite{Mi}.
We use the previous corollary to give a quick new proof in Theorem \ref{shentheorem} of a result by Shen \cite{S}, which characterises when higher Auslander algebras are of finite representation-type.

Our next application deals with two articles, written by Ringel \cite{R} and Zito \cite{Z}.
Namely in \cite{Z}, Zito showed that an Auslander algebra is tilted if and only if $\pdim \tau \Omega^1(D(A)) \leq 1$, which raised the question for an explicit classification of tilted Auslander algebras.
In \cite{R}, Ringel studied QF-1 algebras of global dimension at most 2 and showed that they have positive dominant dimension as a main result, which raises the question of classifying QF-1 algebras of global dimension at most 2. Here QF-1 (short for Quasi-Frobenius-1) algebras are a generalisation by Thrall \cite{T} of the classical quasi-Frobenius algebras
and are by definition those algebras $A$ such that all faithful left and right $A$-modules are balanced. 
We give an explicit classification of tilted Auslander algebras and also of QF-1 algebras of global dimension at most 2 with the surprising result that both classes coincide:
\begin{thmintroduction}[see Theorem \ref{theoremB}] \label{thmB}
Let $K$ be an algebraically closed field.
The following are equivalent:
\begin{enumerate}
\item $A$ is a tilted Auslander algebra.
\item $A$ is a QF-1 algebra of global dimension at most 2.
\item $A$ is the Auslander algebra of a path algebra of Dynkin type $A_1, A_2$ or $A_3$ with non-linear orientation.
\item $A$ is the Auslander algebra of a bound quiver algebra $B$ such that every indecomposable $B$-module is projective or injective.
\end{enumerate}
\end{thmintroduction}

The previous theorem classifies QF-1 higher Auslander algebras for global dimension 2.
We pose the general classification as a problem:
\begin{problem}
Classify the higher Auslander algebras that are QF-1.
\end{problem}
In general, checking whether a given algebra $A$ is QF-1 is very complicated as it is a condition about all $A$-modules. 
Our third main result characterises higher Auslander QF-1 algebras in terms of two local conditions. The strategy behind the proof is based on torsion theory.
\begin{thmintroduction}[see Theorem \ref{TheoremC}] \label{thmC}
    Let $A$ be a higher Auslander algebra of global dimension $g$. $A$ is QF-1 if and only the following two conditions are satisfied:
    \begin{enumerate}
        \item The injective dimension of $\underline{A}$ as a right $A$-module is at most $g-1$;
        \item If $Ae$ is not injective for a primitive idempotent $e$ and $fA$ is not injective for a primitive idempotent, then $fAe=0$.
    \end{enumerate}
\end{thmintroduction}

In particular, this generalises one direction of Theorem B in the sense that every higher Auslander QF-1 algebra is the endomorphism algebra of an $n$-cluster tilting modules formed by only projective or injective modules (see Corollary \ref{cordot13}). Theorem C furthermore gives the first finite test to check whether a given higher Auslander algebra is QF-1. In the Appendix \ref{appendix}, Theorem C is implemented as a program for the GAP package \cite{QPA} to test the QF-1 property for higher Auslander algebras, see the demonstration in the highly non-trivial Example \ref{nontrivialexampleQF1highaus}.
Moreover, in Section \ref{Examples}, we present examples of higher Auslander algebras of global dimension $g$ which are QF-1 for every $g\geq 2$.

\section{Preliminaries}
In this section, we summarise the concepts and terminology needed.
Throughout this paper, all algebras are assumed to be finite-dimensional $K$-algebras over a field $K$ and all modules are assumed to be finitely generated right modules unless stated otherwise.

For introductions to representation theory and homological algebra of finite-dimensional algebras, we refer for example to \cite{ARS} and \cite{ASS}.
We always assume that our algebras are connected without loss of generality, which simplifies some of our definitions and statements.
Let $A$ be a finite-dimensional algebra over a field $K$. By $\rmod A$ we denote the category of all finitely generated right $A$-modules. Given a module $X$, we write $\add X$ to denote the additive closure of $X$ and $\proj A$ to denote $\add A$. Given a map $f\colon M\rightarrow N$ between two modules, we denote its image by $\im f$. By $\Tr_M N$ we denote the trace of $M$ in $N$. By $D$ we mean the standard duality $D\colon \rmod A\leftrightarrow \rmod A^{op}$, where $A^{op}$ denotes the opposite algebra of $A$. Given a module $X$, we write $I(X)$ to denote the injective hull of $X$. We write $\soc X$ (resp. $\top X$) to denote the socle (resp. the top) of $X$ and we denote by $\transpose X$ the transpose of $X$. 
By $\nu=D\Hom_A(-, A)$, we denote the Nakayama functor and by $\nu^{-1}$ its right adjoint, the inverse Nakayama functor. We sometimes write $\nu_A$ and $\nu^{-1}_A$, respectively, in order to emphasize the underlying algebra.

\subsection{Higher Auslander algebras and cluster tilting modules}
For $n \geq 1$ let $\tau_n(X):= \tau \Omega^{n-1}(X)$ be the \emph{higher Auslander-Reiten translate} of a module $X$ and $\tau_n^{-1}=\tau^{-1} \Omega^{-(n-1)}(X)$ the inverse higher Auslander-Reiten translate, where $\Omega^n (X)$ is the $n$-syzygy of $X$ and $\Omega^{-n}(X)$ is the $n$-cosyzygy of $X$. Sometimes, we write $\tau_n^A(X)$ to emphasize the underlying algebra.  For every indecomposable non-projective $A$-module $X$ in $^{\perp_{n-1}}A$ we have $\tau_n^{-1}(\tau_n X)\cong X$ (see \citep[1.4.1]{Iya}).
Here, the perpendicular category of $M$, ${}^{\perp_n} M$, is defined as the subcategory $${}^{\perp_n} M=\{X\in \rmod A\colon \Ext_A^i(X, M)=0, i=1, \ldots, n\}.$$ Dually, one can consider the perpendicular category $M^{\perp_n}$ for every $n\geq 1$.
We will use the usual notation to denote the stable Hom functor by $\underline{\Hom}$ and the costable Hom functor by $\overline{\Hom}$.

We note the following from \cite[Theorem 1.5]{Iya}.
\begin{theorem} \label{iyama lemma}
Let $X \in {}^{\perp n} A$ and $Y \in D(A)^{\perp n}$. For any $1 \leq i \leq n$, there exist functorial isomorphisms for any $A$-module $Z$:
\begin{enumerate}
\item $\Ext_A^{n+1-i}(X,Z) \cong D \Ext_A^i(Z, \tau_{n+1}(X))$ and $\underline{\Hom}_A(X,Z) \cong D \Ext_A^{n+1}(Z, \tau_{n+1}(X))$.
\item $\Ext_A^{n+1-i}(Z,Y) \cong D \Ext_A^i(\tau_{n+1}^{-1}(Y),Z)$ and $\overline{\Hom}_A(Z,Y) \cong D \Ext_A^{n+1} ( \tau_{n+1}^{-1}(Y),Z)$.
\end{enumerate}
\end{theorem}
Given a module $M$, we denote by $\pdim M$ (resp. $\idim M$) the projective (resp. injective) dimension of $M$. By $\domdim M$ or $\domdim_A M$ (resp. $\codomdim M$) we denote the dominant dimension (resp. codominant dimension) of $M$. For a background on these homological invariants, we refer to \cite{Tac2, Yam}. By $\gldim A$ we denote the global dimension of $A$. 

 Recall that an  $A$-module $M$ is called \emph{$n$-cluster tilting}, or just cluster tilting if the specific $n$ does not play a role, for some $n\in \mathbb{N}$ if $\add M= {}^{\perp_{n-1}}M=M^{\perp_{n-1}}$. These modules are also known as maximal $(n-1)$-orthogonal modules in \cite{Iya}, see \cite{Iya2}. In particular, an $n$-cluster tilting module is a generator-cogenerator and $\tau_n M\in \add M$ and $\tau_n^{-1}M\in \add M$ (see \citep[2.3]{Iya}). A generator is a module that contains every projective indecomposable as direct summand, while a cogenerator is defined dually. An algebra has finite representation-type if there are only finitely many isomorphism classes of finite-dimensional indecomposable modules. For those algebras, one can associate an algebra, called Auslander algebra, using the Auslander correspondence. The Auslander algebra of a representation finite-type algebra is the endomorphism algebra of the direct sum of all non-isomorphic indecomposable modules. A subcategory $\mathcal{C}$ of $\rmod A$ is of \emph{finite-type} if there are only finitely many non-isomorphic indecomposable modules in $\mathcal{C}$.

An algebra $A$ is a \emph{higher Auslander algebra} if $\gldim A= \domdim A \geq 2$ or $A$ is semisimple. An Auslander algebra is simply a higher Auslander algebra of global dimension 2.
A recent highlight in representation theory is the result of Iyama that higher Auslander algebras are in bijective correspondence to cluster tilting modules, we refer for example to \cite{Iya2} for details. The correspondence is given by associating to a cluster tilting module $M$, the endomorphism algebra $\End_A(M)$ and to a higher Auslander algebra $B$ the endomorphism ring of $P$, when $P$ is the direct sum of all indecomposable projective-injective modules. This correspondence is called the higher Auslander correspondence as it generalises the classical result of Auslander, which states that Auslander algebras are in bijective correspondence with representation-finite algebras. An algebra $A$ is called \emph{$n$-representation-finite} if it has an $n$-cluster tilting module and $\gldim A\leq n$.

\subsection{QF-1 algebras}
An algebra $A$ is called \emph{QF-1} if every faithful $A$-module is balanced. An $A$-module $M$ is called \emph{balanced} if the canonical homomorphism of algebras $A\rightarrow \End_A(M)$ is surjective. For more information on QF-1 algebras and classification of QF-1 Nakayama algebras, we refer to \cite{Tac2} and \cite{RT}.
Morita in \cite{zbMATH03171195} proved the following important characterisation of QF-1 algebras among algebras with dominant dimension at least one:
\begin{theorem} \label{moritatheorem}
Let $A$ be an algebra with dominant dimension at least one.
Then $A$ is QF-1 if and only if $A$ has dominant dimension at least two and every indecomposable $A$-module has dominant dimension at least one or codominant dimension at least one, that is:
$$\domdim M + \codomdim M \geq 1$$
for all indecomposable modules $M$. \label{thmonedotten}

\end{theorem}
For a modern proof see for example \cite[Section 3.5]{Yam}. 

As a consequence of Theorem \ref{moritatheorem}, we obtain that the property of being QF-1 is left-right symmetric for algebras of positive dominant dimension, that is, $A$ is QF-1 if and only if $A^{op}$ is QF-1 whenever $\domdim A\geq 1$.

\subsection{Torsion pairs}
For a background on the study of torsion pairs, we recommend \citep[Chapter VI]{ASS}. A pair ($\mathcal{T}$, $\mathcal{F}$) of full subcategories of $\rmod A$ is called a \emph{torsion pair} if $$\mathcal{F}=\{X\in \rmod A\colon \Hom_A(M, X)=0, \forall M\in \mathcal{T}\}$$ and $$\mathcal{T}=\{X\in \rmod A\colon \Hom_A(X, N)=0, \ \forall N\in \mathcal{F} \}.$$ $\mathcal{T}$ is called the \emph{torsion class} and $\mathcal{F}$ is usually called the \emph{torsion-free class}. The torsion class is always closed under images, direct sums and extensions, while the torsion-free class is closed under submodules, direct products and extensions. Every torsion pair comes equipped with a functor $t\colon \rmod A\rightarrow \rmod A$ called \emph{torsion radical} such that for every indecomposable $A$-module $M$ there exists a short exact sequence $0\rightarrow tM\rightarrow M\rightarrow M/tM\rightarrow 0$ with $tM\in \mathcal{T}$ and $M/tM\in \mathcal{F}$. Such an exact sequence is unique. As a consequence, every simple module lies either in the torsion class or in the torsion-free class. A torsion pair $(\mathcal{T}, \mathcal{F})$ is called \emph{splitting} if every indecomposable module lies either in $\mathcal{T}$ or in $\mathcal{F}$. In such a case, we also say that the torsion pair splits. 

For a given $M\rmod A$, we define 
$$\gen M =\{X\in \rmod A\colon \ \text{there exists an epimorphism} \ M^n\twoheadrightarrow X \text{ for some } n\in \mathbb{N}\}$$ and
$$\cogen M =\{X\in \rmod A\colon \ \text{there exists a monomorphism} \ X\hookrightarrow M^n \text{ for some } n\in \mathbb{N}\}.$$
In our cases of interest, the torsion class is given by $\gen M$ while the torsion-free class is by $\cogen N$ for some modules $M$ and $N$.

\section{A non-commutative generalisation of the Auslander-Buchsbaum formula for higher Auslander algebras}

In this section, the aim is to prove the main theorem \ref{theoremA}, in particular, a new formula for the projective dimension (Theorem \ref{domdimprojdimformulas}(2)) and a kind of Poincar\'e duality involving the stable and costable module for higher Auslander algebras (Theorem \ref{thm3dot7}). In the process, we obtain several properties of the stable and costable module. As applications, in this section, we prove a non-commutative analogue of the Auslander-Buchweitz formula (Corollary \ref{cor3dot8}). We then use this formula to prove that higher Auslander algebras are of finite representation-type if and only if they have only finitely many indecomposable modules with zero dominant dimension (Theorem \ref{shentheorem}).

For the convenience of the reader, we start with a technical lemma which is a zero analogue of \citep[Proposition 2.6]{APT}.
\begin{lemma}\label{casezeroAPT}
Let $x$ be an idempotent of $A$ and $M\in \rmod A$. Then $\Hom_A(A/AxA, M)=0$ if and only if $I(M)\in \add D(Ax)$.
\end{lemma}
\begin{proof}
    Observe that $\Hom_A(A/AxA, D(Ax))\cong \Hom_A(Ax, D(A/AxA))\cong xD(A/AxA)=0$. So, if $I(M)\in \add D(Ax)$, then $\Hom_A(A/AxA, M)\subset \Hom_A(A/AxA, I(M))=0$.

    Conversely, assume that $\Hom_A(A/AxA, M)=0$. Observe that $\Hom_A(A/AxA, I(M))$ is the injective hull of $\Hom_A(A/AxA, M)$ (see for instance \citep[Proposition 1.1 and discussion above it]{APT}). Thus, $\Hom_A(A/AxA, I(M))=0$. Let $l$ be an idempotent so that $D(Al)$ is a direct summand of $I(M)$. Thus, $$0=\Hom_A(A/AxA, D(Al))\cong \Hom_A(Al, D(A/AxA))\cong D(A/AxA)\cdot l\cong D(l\cdot A/AxA).$$
    Hence, $l\in AxA$ and we can write $l=yxw$ for some elements $y, w\in A$. So, the map $Ax\rightarrow Al$, given by $ax\mapsto axwl$, is surjective. It follows that $Al\in \add Ax$ and thus $D(Al)\in \add D(Ax)$. Moreover, $I(M)\in \add D(Ax).$
\end{proof}

\begin{theorem} \label{domdimprojdimformulas}
Let $A$ be a finite-dimensional algebra and $M\in \rmod A$.
\begin{enumerate}
\item $\domdim M:= \inf \{ n \geq 0 \mid \Ext_A^n(\underline{A},M) \neq 0 \},$ if $M$ is not projective-injective.
\item Assume that $M$ has finite projective dimension over $A$. Then, $$\pdim M= \sup \{ n \geq 0 \mid \Ext_A^n(M,\overline{A}) \neq 0 \}.$$ 
\end{enumerate}
\end{theorem}
\begin{proof}
\begin{enumerate}
    \item This is  obtained as a special case of \cite[Proposition 2.6]{APT} by setting $P=Af$ in the proposition of the article by Auslander, Platzeck and Todorov and Lemma \ref{casezeroAPT}.
    \item 
Of course, $$\pdim_A M\geq \sup \{ n \geq 0 \mid \Ext_A^n(M,A/AeA) \neq 0 \},$$ since $\Ext_A^{i+\pdim M}(M, -)=0$ for $i>0$. So, there is nothing more to show for projective $A$-modules $M$. 

Assume that $\pdim_A M=1$. We want to show $\Ext_A^1(M, A/AeA)\neq 0$. Suppose by contradiction that $\Ext_A^1(M, A/AeA)= 0$.
By the Auslander-Reiten formula (see for instance \citep[IV, Corollary 2.14]{ASS} or \citep[IV, Corollary 4.7]{ARS}),
$$0=\Ext_A^1(M, A/AeA)\cong D\Hom_A(A/AeA, \tau M).$$ Hence, $\Hom_A(A/AeA, \tau M)=0.$ By Lemma \ref{casezeroAPT}, this implies $I(\tau M)\in \add D(Ae)$ since $D(Ae)$ is the injective hull of $\top eA$ and $AeA$ is the trace ideal of $eA$ in $A$. Let \begin{equation}
    0\rightarrow P_1\rightarrow P_0\rightarrow  M\rightarrow 0 \label{eq1}
\end{equation} be the minimal projective resolution of $M$. Then, applying $D\Hom_A(-, A)$ we obtain that \mbox{$0\rightarrow \tau M\rightarrow \nu P_1\rightarrow \nu P_0$} is the beginning of the minimal injective coresolution of $\tau M$. Hence, we get the following implications
\begin{align*}
    \nu P_1\in \add D(Ae) &\implies \Hom_A(P_1, A)\in \add Ae \\&\implies 
    P_1\cong \Hom_A(\Hom_A(P_1, A), A)\in \add \Hom_A(Ae, A)
    \\&\implies P_1\in \add eA. 
\end{align*} But, this implies that $P_1$ is injective and so (\ref{eq1}) would split contradicting $\pdim_A M$ being exactly one. So, this shows that if $\pdim_A M=1$, then $\Ext_A^1(M, A/AeA)\neq 0$.

Assume now that $\pdim_A M=k$ for some natural number $k$. Then, $\pdim_A \Omega^{k-1}(M)=1$ and so by the previous argument 
$$\Ext_A^k(M, A/AeA)\cong \Ext_A^1(\Omega^{k-1}(M), A/AeA)\neq 0.$$ So, $\sup \{ n \geq 0 \mid \Ext_A^n(M,A/AeA) \neq 0 \}\geq n=\pdim_A M$, and thus (2) holds.
\end{enumerate}
\end{proof}

We now prepare for the proof of part (3) of Theorem \ref{theoremA}, which is the most complicated part. The next lemma justifies our choice of notation for $A/AfA$ and $A/AeA$.

\begin{lemma}\label{notationstablemodules}
	Let $A$ be a finite-dimensional algebra with dominant dimension at least two. Let $e$ and $f$ be idempotents of $A$ such that $Af$ and $eA$ are faithful projective-injective $A$-modules. Then, $$A/AfA\cong \underline{\End}_{fAf}(Af) \text{ and } A/AeA\cong \overline{\End}_{fAf}(Af).$$
\end{lemma}
\begin{proof}
	It is well known that $eAe\cong fAf$ and $eAf$ is an injective left $eAe$-module and an injective right $fAf$-module. By the Morita-Tachikawa correspondence, there are double centraliser properties $A\cong \End_{eAe}(eA)\cong \End_{fAf}(Af)^{op}$. This means that the pairs $(A, Ae)$, $(A, fA)$ are covers of $eAe\cong fAf$. Hence, the functor given by right multiplication by $f$ restricts to equivalences $\add eA\rightarrow \inj fAf$ and $\add fA\rightarrow \proj fAf$ with quasi-inverse given by $\Hom_{fAf}(Af, -)$. 
	It follows that each homomorphism $Af\rightarrow Af$ that factors through $\inj fAf$ corresponds uniquely to a homomorphism $A\rightarrow A$ that factors through $eA$. Hence, $\overline{\End}_{fAf}(Af)\cong A/AeA$. Analogously, each homomorphism $Af\rightarrow Af$ that factors through $\proj fAf$ corresponds uniquely to a homomorphism $A\rightarrow A$ that factors through $fA$. So, $\underline{\End}_{fAf}(Af)\cong A/AfA$.
\end{proof}

\begin{remark}
    We observe that under the notation of Lemma \ref{notationstablemodules}, we have $\underline{A^{op}}=A/AeA$ and $\overline{A^{op}}=A/AfA$ as left $A$-modules.
\end{remark}

Next, we give an example that shows that the assumption of finite projective dimension in \Thmref\ref{domdimprojdimformulas}(2) cannot be dropped.
\begin{example}
Let $A$ be the bound quiver $K$-algebra
$$\begin{tikzcd}
	\arrow[loop, distance=-5em, out=35, in=-35, "\gamma"] 1 \arrow[r, "\beta"] & 2 \arrow[loop, distance=3em, out=35, in=-35, "\alpha"]
\end{tikzcd} \quad \text{with relations} \quad 0=\gamma^2=\alpha^2=\gamma\beta-\beta\alpha.$$ Then, $A$ is an Iwanaga-Gorenstein algebra with Gorenstein dimension one and also with dominant dimension one. Let $e_i$ be the idempotent associated with the vertex $i$. Then, $e_1A$ is the faithful projective-injective right $A$-module and so $\overline{A}=A/Ae_1A$ is isomorphic to the projective $e_2A$. Observe that the socle $S=\soc(e_2 A)$ of $e_2A$ is not projective and satisfies $\Ext_A^1(S,A)=0$ and thus $\Ext_A^i(S,A)=0$ for all $i >0$ since $\idim A=1$. Hence, the projective dimension of $\soc(e_2A)$ is infinite, but $\sup\{n\geq 0\colon \Ext_A^n(\soc(e_2A), e_2A)\neq 0\}=0.$
\end{example}

We now identify the stable module with the image of the regular module under the inverse higher Auslander-Reiten translate.

\begin{lemma} \label{formulahighertauinverse}
Let $A$ be a higher Auslander algebra of global dimension $g$. Let $f$ be an idempotent of $A$ such that $Af$ is a faithful projective-injective $A$-module. Then, 
$\add \tau_g^{-1}(A)=\add D \Ext_A^g(D(A),A) =\add \underline{A}$, where $D\Ext_A^g(D(A),A)$ inherits the right $A$-module structure from the left $A$-module structure of $D(A)$.

\end{lemma}
\begin{proof}
We have $\tau_g^{-1}(A) \cong \tau^{-1} \Omega^{-(g-1)}(A).$
Now note that the modules $\Omega^{-(g-1)}(A)$ and $\Omega^1(D(A))$ have injective dimension at most one and the same additive closure (see for instance \citep[Corollary 6.5]{CPsa}). Thus, $$\tau^{-1}(\Omega^1(D(A)))\cong D \Ext_A^1(D(A),\Omega^1(D(A))) \text{ and } D \Ext_A^1(D(A), \Omega^{-(g-1)}(A))\cong D \Ext_A^g(D(A),A)$$ have the same additive closure. Here, we used the formula $\tau^{-1}(M) \cong D \Ext_A^1(D(A),M)$ valid for modules $M$ with injective dimension at most one (see for example \cite[Chapter IV. Corollary 2.14]{ASS}).
Now we want to show that $\tau^{-1}(\Omega^1(D(A))) \cong A/AfA$ or equivalently 
$\tau(A/AfA) \cong \Omega^1(D(A))$.
Take a minimal projective presentation of $A/AfA$:
$$ (fA)^m \rightarrow A \rightarrow A/AfA \rightarrow 0.$$
Here the kernel for the canonical projection map $A \rightarrow A/AfA$ is $AfA=\Tr_{fA}(A)$.
This gives the following exact sequence by the definition of the Auslander-Reiten translate $\tau$:
$$0 \rightarrow \tau(A/AfA) \rightarrow \nu((fA)^m) \rightarrow D(A) \rightarrow \nu(A/AfA).$$
But we have $\nu(A/AfA)= D\Hom_A(A/AfA,A)=0$, since $A$ has dominant dimension at least one using (1) of Theorem \ref{domdimprojdimformulas}.
Thus the exact sequence 
$$0 \rightarrow \tau(A/AfA) \rightarrow \nu((fA)^m) \rightarrow D(A) \rightarrow 0$$
gives that $\tau(A/AfA)$ is the first syzygy module of $D(A)$.
\end{proof}

The next proposition shows that the costable module is the image of the regular module under the higher Auslander-Reiten translate.

\begin{proposition}  \label{taustablehomformula}
		Let $B$ be an algebra with $d$-cluster tilting module $M$
		and $A=\End_B(M)$ the corresponding higher Auslander algebra of dominant dimension $d+1$.
		Then for $X \in \add(M):$
		$$\tau_{d+1}^A(\underline{\Hom}_B(M,X)) \cong \overline{\Hom}_B(M,\tau_d^B(X)).$$
		In particular, $\tau_{d+1}^A(\underline{A})\cong \overline{A}.$
\end{proposition}
\begin{proof}
	Consider the exact sequence $0 \rightarrow \Omega^1(X) \rightarrow P_0 \rightarrow X \rightarrow 0$ with $P_0$ the projective cover of $X$ and set $X_1=\Omega^1(X)$. Observe that all maps $M\rightarrow X$ that factor through a projective module must also factor through $P_0\twoheadrightarrow X$, and thus we have the following exact sequence of left $A$-modules \begin{align}
		0\rightarrow \Hom_B(M, X_1)\rightarrow \Hom_B(M, P_0)\rightarrow \Hom_B(M, X)\rightarrow \underline{\Hom}_B(M, X)\rightarrow 0.
	\end{align}
	By assumption, $A$ has global dimension at most $d+1$, thus the projective dimension of $\Hom_B(M, X_1)$ is at most $d-1$. Hence, by projectivization, there are modules $M_i\in \add M$, $i=1, \ldots, d$, so that 
	\begin{align}
		0\rightarrow \Hom_B(M, M_d)\rightarrow \cdots \rightarrow \Hom_B(M, M_1) \rightarrow \Hom_B(M, P_0)\rightarrow \Hom_B(M, X)\twoheadrightarrow \underline{\Hom}_B(M, X). \label{eq13}
	\end{align} is a minimal projective resolution of $ \underline{\Hom}_B(M, X)$.
Since $M$ is a generator of $\rmod B$, all minimal right $\add M$-approximations are surjective and $-\otimes_A M$ is an exact functor, and thus the resolution (\ref{eq13}) is induced by an exact sequence
\begin{align}
	0\rightarrow M_d\rightarrow\cdots\rightarrow M_1\rightarrow M_0\rightarrow X\rightarrow 0, \label{eq14}
\end{align}where $M_0:=P_0$. We claim that $\tau_{d+1}^A\underline{\Hom}_B(M, X)\cong D\Ext_B^d(X, M)$. By applying $\Hom_A(-, A)=(-)^*$ to (\ref{eq13}) we get the exact sequence
\begin{align}
	0\rightarrow \Omega^{d}(\underline{\Hom}_B(M, X))^*\rightarrow \Hom_B(M_{d-1}, M)\rightarrow \Hom_B(M_d, M)\rightarrow \transpose \Omega^{d}(\underline{\Hom}_B(M, X))\rightarrow 0.
\end{align}
since $$\Hom_B(M, M_i)^*=\Hom_A(\Hom_B(M, M_i), A)\cong \Hom_A(\Hom_B(M, M_i), \Hom_B(M, M))\cong \Hom_B(M_i, M).$$ Therefore, $\tau_{d+1}^A\underline{\Hom}_B(M, X)$ is defined as the following kernel
\begin{align}
	0\rightarrow \tau_{d+1}^A\underline{\Hom}_B(M, X) \rightarrow D\Hom_B(M_d, M)\rightarrow D\Hom_B(M_{d-1}, M).
\end{align} Now, applying $D\Hom_B(-, M)$ to (\ref{eq14}) yields the exact sequence
$$D\Ext_B^1(M_{d-1}, M)\rightarrow D\Ext_B^1(X_{d-1}, M)\rightarrow D\Hom_B(M_d, M)\rightarrow D\Hom_B(M_{d-1}, M),$$
where $X_{i}$ denotes the image of $M_i\rightarrow M_{i-1}$ for $i=1, \ldots, d$ and $X_0:=X$. If $d=1$, then $M_{d-1}=P_0$ is projective over $B$ and so $\Ext_B^1(M_{d-1}, M)=0$. If $d>1$, then since $M$ is a $d$-cluster tilting module we obtain that $\Ext_B^1(M_{d-1}, M)=0$ as well. It follows that $\tau_{d+1}^A\underline{\Hom}_B(M, X) \cong D\Ext_B^1(X_{d-1}, M)$. 
If $d=1$, then the Auslander-Reiten formula yields that  $D\Ext_B^1(X_{d-1}, M)=D\Ext_B^1(X, M)\cong \overline{\Hom}_B(M, \tau^B X)$ and the result follows for $d=1$.
Assume now that $d>1$. For this case, we have $\Ext_B^{2}(X_{d-2},M) \cong \Ext_B^{1}(X_{d-1},M)$, by the long exact sequence obtained by applying $\Hom_B(-,M)$ to the exact sequence $0 \rightarrow X_{d-1} \rightarrow M_{d-2} \rightarrow X_{d-2} \rightarrow 0.$ In the same way, induction gives that $\Ext_B^l(X_{d-l},M) \cong \Ext_B^{l-1}(X_{d-(l-1)},M)$ and thus $\Ext_B^1(X_{d-1},M) \cong \Ext_B^d(X,M).$ Since $M$ is a $d$-cluster tilting module, we have that $DB, \tau_d X\in \add M$ and, in particular, $\tau_d X\in DB^{\perp_{d-1}}$. By Theorem  \ref{iyama lemma}, $\overline{\Hom}_B(M, \tau_d X)\cong D\Ext_B^d(\tau_d^{-1}(\tau_d(X)), M)$. Now, write $X=X'\oplus X''$, where $X''$ stands for the maximal direct summand of $X$ which is a projective $B$-module. Observe that $B\in \add M=M^{\perp_{d-1}}$, and thus $X'\in {}^{\perp_{d-1}} B$. So, $\tau_d^{-1}(\tau_d(X))\cong \tau_d^{-1}(\tau_d(X'))\cong X'$ since $X'$ has no trivial projective modules as direct summands. Hence,
$\overline{\Hom}_B(M, \tau_d X)\cong D\Ext_B^d(X', M)\cong D\Ext_B^d(X, M)\cong \tau_{d+1}^A \underline{\Hom}_B(M, X)$.
\end{proof}

We are now ready to prove (3) of Theorem \ref{theoremA}:
\begin{theorem}\label{thm3dot7}
    Let $A$ be a higher Auslander algebra with global dimension $d$ and $n$ an integer. If $M$ is a non-projective-injective $A$-module, then
$$\Ext_A^n(\underline{A},M) \cong D \Ext_A^{d-n}(M,\overline{A}).$$
\end{theorem}
\begin{proof}
If $n>d$ or $n<0$, then both terms are zero since $\gldim A=d$.

Since $A$ has dominant dimension $d$, we have $\underline{A} \in {}^{\perp d-1} A$ by (1) of Theorem \ref{domdimprojdimformulas}. Assume  that $n\in \{1, \ldots, d-1\}$ and let $M$ be a non-projective-injective $A$-module.
Now we can use Theorem \ref{iyama lemma}, to obtain: 
$$\Ext_A^n(\underline{A},M) \cong D \Ext_A^{d-n}(M,\tau_d(\underline{A})) \cong D \Ext_A^{d-n}(M,\overline{A}),$$
where we used $\tau_d(\underline{A}) \cong \overline{A}$ as was established in Proposition \ref{taustablehomformula}.
Now assume that $n=0$. By Theorem \ref{iyama lemma}, we obtain that 
$D\Ext_A^d(M, \overline{A})\cong \underline{\Hom}_A(\underline{A}, M). $ By Lemma \ref{casezeroAPT}, $\Hom_A(\underline{A}, A)=0$ and so $\underline{\Hom}_A(\underline{A}, M)\cong \Hom_A(\underline{A}, M).$ It remains to show the case $n=d$.  Since $\codomdim DA=\domdim A=g$ the projective cover of $DA$ is in the additive closure of $eA$. So, $\Hom_A(DA, \overline{A})=0$ because $\Hom_A(eA, \overline{A})\cong A/AeA\cdot e=0.$ Since $\pdim DA=g$ we have $\overline{A}\in DA^{\perp_{d-1}}.$ Finally, by Theorem \ref{iyama lemma}(2), it follows that $\Hom_A(M, \overline{A})\cong \overline{\Hom}_A(M, \overline{A})\cong D\Ext_A^d(X, M)\cong D\Ext_A^d(\underline{A}, M),$
where $X$ is the maximal direct summand of $\underline{A}$ not having non-indecomposable projective direct summands. The second isomorphism is a consequence of Proposition \ref{taustablehomformula} since $X\cong \tau_d^{-1}(\tau_d(\underline{A}))\cong \tau_d^{-1}(\overline{A}).$
\end{proof}

We now obtain the non-commutative Auslander-Buchsbaum formula for higher Auslander algebras as a direct consequence:
\begin{corollary}\label{cor3dot8}
Let $A$ be a higher Auslander algebra of global dimension $d$ and $M$ a module that is not projective-injective. Then, the following equalities hold.
\begin{enumerate}
    \item $\pdim M + \domdim M=\gldim A;$
    \item $\idim M + \codomdim M= \gldim A$.
\end{enumerate}
\end{corollary}
\begin{proof}
It is enough to prove the first one since the second follows from the first using that the opposite of a higher Auslander algebra is again a higher Auslander algebra of the same global dimension. Let $M$ be an $A$-module that is not projective-injective.

Assume that $M$ has dominant dimension $r$. Hence, $\Ext_A^i(\underline{A},M)=0$ for $i=0,1,...,r-1$ and $\Ext_A^r(\underline{A},M) \neq 0$, which  gives us that $\Ext_A^{d-r}(M, \overline{A}) \neq 0$ and $\Ext_A^i(M, \overline{A})=0$ for $i>d-r$ using (3) of Theorem \ref{theoremA} and $d=\gldim A$. Indeed,
for $l=1, \ldots, r$, we have
$$\Ext_A^{l+d-r}(M, \overline{A})\cong D\Ext_A^{r-l}(\underline{A}, M)=0. $$
Thus by Theorem \ref{domdimprojdimformulas} (2), we get that $\pdim M=d-r$ and thus $\domdim M + \pdim M = d = \gldim A.$
\end{proof}

In general, we have the following relation between injective (resp. projective) and dominant (resp. codominant) dimensions.

\begin{proposition} \label{inequpdim}
Let $A$ be a finite-dimensional algebra.
\begin{enumerate}
    \item If $M$ is a non-injective $A$-module, then $\idim M > \domdim \tau^{-1}(M).$
    \item If $M$ is a non-projective $A$-module, then $\pdim M > \codomdim \tau(M).$
\end{enumerate}
\end{proposition}
\begin{proof}
We prove (1) as the proof of (2) is dual.
Set $\domdim \tau^{-1}(M) \geq t$, then 
$$\tau^{-1}(M) = \Omega^t_A(\Omega^{-t}_A(\tau^{-1}(M))).$$
Since $M$ is non-injective, 
\begin{align}
0\neq \operatorname{Ext}^1_A(\tau^{-1}(M),M) & 
= \operatorname{Ext}^1_A(\Omega^t_A(\Omega^{-t}_A(\tau^{-1}(M))), M)\notag\\
& \cong \operatorname{Ext}^{t+1}_A(\Omega^{-t}_A(\tau^{-1}(M)),M)\notag
\end{align}
This gives us: 
$$\operatorname{idim} M \geq t+1 > t \geq \operatorname{domdim} \tau^{-1}(M).$$ 
\end{proof}

The main result of \cite{S} states that a higher Auslander algebra of global dimension $g$ is of finite representation-type if and only the subcategory of modules of projective dimension $g$ is of finite-type.

We give an extension and an alternative proof of this result using the non-commutative analogue of the Auslander-Buchsbaum formula: 
\begin{theorem} \label{shentheorem}
Let $A$ be a higher Auslander algebra of global dimension $d$. The following assertions are equivalent.
\begin{enumerate}
    \item $A$ is of finite representation-type.
    \item The subcategory $W$ of $A$-modules of projective dimension $d$ is of finite-type.
    \item The subcategory $I$ of $A$-modules of injective dimension $d$ is of finite-type.
    \item The subcategory $D$ of $A$-modules of dominant dimension zero is of finite-type.
    \item The subcategory $C$ of $A$-modules of codominant dimension zero is of finite-type.
\end{enumerate}
\end{theorem}
\begin{proof}
If $A$ is of finite representation-type, then clearly $W$ is of finite-type.
Now assume that $W$ is of finite-type.
Let $X$ be an indecomposable $A$-module.
We want to show that there are only finitely many such $X$ and thus we can assume that $X$ is not projective-injective as there are only finitely many projective-injective indecomposable $A$-modules.
If $\pdim X=d$, then there are only finitely many such $X$.
Thus we can assume that $\pdim X<d$.
Then we know that $\pdim X+ \domdim X=d$.
Let $s:=\domdim X= d- \pdim X>0$.
Then we have an exact sequence:
$$0 \rightarrow X \rightarrow I^0 \rightarrow I^1 \rightarrow \cdots \rightarrow I^{s-1} \rightarrow \Omega^{-s}(X) \rightarrow 0,$$
where all $I^i$ are projective-injective $A$-modules. 
Now $\Omega^{-s}(X)$ has dominant dimension zero, or else $X$ would have dominant dimension strictly larger than $s$.
Thus $\pdim \Omega^{-s}(X)=d- \domdim \Omega^{-s}(X)=d.$
We also have $X\cong \Omega^s(\Omega^{-s}(X)),$ since all terms $I^i$ are projective-injective.
This means that $X$ is uniquely determined as the $s$-th syzygy module of the module $\Omega^{-s}(X)$ of projective dimension $d$.
But since there are only finitely many indecomposable modules of projective dimension $d$, there are also only finitely many $X$ of projective dimension $d-s$ for any $s$ with $0 \leq s \leq d$. Thus (1) is equivalent to (2).

By Corollary \ref{cor3dot8}, the class of modules of dominant dimension zero coincide with the class of modules of projective dimension $d$. Then, (2) is equivalent to (4). Analogously, (3) is equivalent to (5). The standard duality gives that $A$ is of finite representation-type if and only if $A^{op}$ is of finite representation-type. It also gives that the subcategory of $A$-modules of projective dimension $d$ is of finite-type if and only if the subcategory of $A^{op}$-modules of injective dimension $d$ is of finite-type. So, the equivalence $(1)\Leftrightarrow (2)$ for $A^{op}$ implies that (1) is equivalent to (3).
\end{proof}

\section{QF-1 algebras and higher Auslander algebras}

In this section, we combine Ringel's work \cite{R} with Zito's work \cite{Z} proving Theorems \ref{thmB} and \ref{thmC} and we provide several characterisations and properties of QF-1 higher Auslander algebras. This way, we create new connections between $g$-quasi-tilted algebras, QF-1 algebras, higher Auslander algebras and $g$-cluster tilting modules having only projective or injective modules as direct summands. 

We call $A$ \emph{$g$-quasi-tilted} if $\gldim A \leq g$ and $\pdim M + \idim M \leq 2g-1$ for all indecomposable $A$-modules $M$.
For $g=2$, this gives the classical \emph{quasi-tilted} algebras, see for example the book \cite{HRS}. Note that for an algebra of global dimension $g$ the condition $\pdim M + \idim M \leq 2g-1$ for all indecomposable $A$-modules $M$ is equivalent to the condition that $\pdim M < g$ or $\idim M < g$ for all indecomposable modules $M$.

We present the following conjecture:
\begin{conjecture} \label{mainconjecture}
Let $A$ be a higher Auslander algebra of \textbf{even} global dimension $g$.
Then the following are equivalent:
\begin{enumerate}
\item $A$ is QF-1.
\item $A$ is $g$-quasi-tilted.
\item $\pdim \tau_g(D(A)) \leq g-1$.
\item $\idim \tau_g^{-1}(A) \leq g-1$. 
\end{enumerate}
\end{conjecture}

Observe that (4) is equivalent to $\idim \underline{A}\leq g-1$ by Lemma \ref{formulahighertauinverse} or even to $\underline{A}$ having positive codominant dimension.
Hence, we claim that not only the homological properties of $\underline{A}$ encode the dominant dimensions in $\rmod A$ but also ring theoretical properties of the algebra $A$ like being QF-1.

The equivalence of (1) and (2) is a consequence of the following result:
\begin{theorem}\label{higherauscharacterisation}
Let $A$ be an algebra of finite global dimension $g$.
Then $A$ is a higher Auslander algebra if and only if $\pdim M+ \domdim M=g$ for every indecomposable $M$ that is not projective-injective.
\end{theorem}
\begin{proof}
Assume that $\pdim M+ \domdim M=g$ for every indecomposable $M$ that is not projective-injective holds, then it holds for all indecomposable projective non-injective modules $P$, giving that $\domdim P=g$ and thus $\domdim A= \gldim A$.
The converse is a direct consequence of Corollary \ref{cor3dot8}.
\end{proof}

\begin{corollary}
In the Conjecture \ref{mainconjecture} we have that (1) and (2) are equivalent and (1) implies (3) and (4).
This also holds for odd global dimension.
\label{cor3dot10}
\end{corollary}
\begin{proof}
Assume that $A$ is a higher Auslander algebra of global dimension $g$ and assume that $A$ is QF-1 and $M$ an is indecomposable $A$-module that is not projective-injective.
Using Theorem \ref{higherauscharacterisation}, we have $\domdim M+\pdim M+\codomdim M+\idim M=2g$, which directly shows the equivalence of (1) and (2).
To see that (1) (which is equivalent to (2)) implies (3) note that $\Ext_A^g(D(A),\tau_g(D(A))\cong\Ext_A^1(\Omega^{g-1}(D(A)),\tau(\Omega^{g-1}(D(A)))$ is non-zero and thus $\idim \tau_g(D(A))=g$. But this implies by (2) that $\pdim \tau_g(D(A))<g$.
Analogously, $\Ext_A^g(\tau_g^{-1}(A), A)\cong \Ext_A^1(\tau^{-1}\Omega^{-(g-1)}(A), \Omega^{-(g-1)}(A))$ is non-zero and $\pdim \tau_g^{-1}(A)=g.$ So, if $A$ is QF-1, we must have $\idim \tau_g^{-1}(A)<g$.
\end{proof}

In general, we have the following global characterisation of QF-1 higher Auslander algebras.

\begin{proposition} \label{prop3dot11}
    Let $A$ be a higher Auslander algebra of global dimension $g$. Let $e$ and $f$ be idempotents of $A$ such that $Af$ and $eA$ are faithful projective-injective $A$-modules.  Then the following are equivalent:
    \begin{enumerate}
        \item $A$ is QF-1.
        \item $A$ is $g$-quasi-tilted.
        \item $\pdim \tau_g(D(A)) \leq g-1$ and $\soc X\cdot AeA\subset \soc eA$ for every indecomposable module $X$ with $\top X\notin \add \top eA$ (i.e. with $\codomdim X=0$).
    \end{enumerate}
\end{proposition}
\begin{proof}
    Assume that $A$ is QF-1. Let $X$ be an indecomposable $A$-module with $\top X\notin \add \top eA$. Then, we obtain that $\codomdim X=0$. Thus, $\domdim X\geq 1$. Therefore $\Tr_{eA}X=\sum_{f\in \Hom_A(eA, X)} \im f\subset X$ has also dominant dimension at least one. Thus, $\soc X\cdot AeA=\soc \Tr_{eA} X\subset \soc eA.$

    Conversely, assume that $\pdim \tau_g(DA)\leq g-1$. Since $A$ is a higher Auslander algebra of global dimension $g$, the module $T=eA\oplus \Omega^{-1}(A)$ is a $1$-tilting-$(g-1)$-cotilting module with \mbox{$\add T=\add eA\oplus \Omega^{g-1}(DA)$,} $\domdim T\geq g-1$ and $\codomdim T\geq 1$. (see for example \cite{CPsa}). This tilting module determines a torsion pair $(\gen T, \cogen \tau T)$ of $\rmod A$.

    \textbf{Claim 1. $\cogen \tau T=\cogen \tau_g(DA)\subset \{X\in \rmod A\colon \domdim X\geq 1\}$. }
    
Indeed, observe that $\add \tau T=\add \tau eA\oplus \tau \Omega^{g-1}(DA)=\add \tau_g(DA).$ By assumption, $\pdim \tau T\leq g-1$. By Corollary \ref{cor3dot8}, $\domdim \tau T\geq 1$, that is, $\tau T$ embeds into a projective-injective module, so the claim follows.

\textbf{Claim 2. $\gen T=\gen eA$}.

Since $\codomdim T\geq 1$ it is clear that every module in $\gen T$ has positive codominant dimension, and so $\gen T\subset \gen eA$. Since $eA$ is a direct summand of $T$, then it is also clear that $\gen eA\subset \gen T.$

By \citep[Section 1]{zbMATH00941221}, $(\gen eA, \cogen \tau_g(DA))$ is a torsion pair of $\rmod A$ with torsion radical given by the trace $\Tr_{eA}(-)$.
    
Let $X$ be an indecomposable module with $\codomdim X=0.$ 
If $X\subset \cogen \tau T$, then by Claim 1, $X$ has positive dominant dimension. Otherwise, there exists an exact sequence $0\rightarrow \Tr_{eA} X\rightarrow X\rightarrow \Tilde{X}\rightarrow 0$ with $\Tr_{eA} X\in \gen eA, \Tilde{X}\in \cogen \tau_g (DA)$. So $\domdim \Tilde{X}\geq 1$. Assume that (3) holds. 
As $\codomdim X=0$, we get that $\top X\notin \add \top eA$. By assumption, $\soc \Tr_{eA}X\subset \soc eA$, that is, $\domdim \Tr_{eA}X\geq 1.$ Thus, the above sequence gives that $\domdim X\geq 1.$ So, $A$ must be QF-1.

The remaining follows by Corollary \ref{cor3dot10}.
\end{proof}

Making use of the above torsion pair, we can derive the following properties of QF-1 higher Auslander algebras. 

\begin{proposition} \label{prop4dot5}
    Let $A$ be a higher Auslander algebra of global dimension $g$.   The following assertions hold.
\begin{enumerate}
    \item If $A$ is QF-1, then $\Ext_A^1(M, N)=0$ for all simple $A$-modules $M$ and $N$ with $\domdim N=0=\codomdim M$.
    \item Let $eA$ be a faithful projective-injective $A$-module. If $A$ is QF-1, then every simple right $A$-module is either in the top of $eA$ or in the socle of $eA$.
    \item If $\pdim \tau_g(D(A)) \leq g-1$ and $\Ext_A^1(M, N)=0$ for all indecomposable $A$-modules with $\domdim N=0=\codomdim M$, then $A$ is QF-1.
\end{enumerate}
\end{proposition}
\begin{proof}
Let $e$ and $f$ be idempotents of $A$ such that $Af$ and $eA$ are faithful projective-injective $A$-modules.

Suppose that $A$ is QF-1. Let $M$ be a simple $A$-module with $\codomdim M=0$ and $N$ a simple $A$-module with $\domdim N=0$. Let $X$ be an extension of $M$ by $N$. Since $M$ and $N$ are simple, $X$ is either an indecomposable module or the direct sum of $M$ and $N$. Assume that the first case occurs. Then, $\domdim_A\soc X=\domdim_AN=0$ and $\codomdim_A\top X=\codomdim_A M$. So $X$ would be an indecomposable module with zero dominant and codominant dimension, contradicting $A$ being QF-1. So, (1) holds.

Assume that $\pdim \tau_g(D(A)) \leq g-1$ and $\Ext_A^1(M, N)=0$ for all indecomposable $A$-modules with $\domdim N=0=\codomdim M$.
    Let $X$ be an indecomposable module with $\codomdim X=0$. Then, there exists an exact sequence $\delta\colon 0\rightarrow \Tr_{eA} X\rightarrow X\rightarrow \Tilde{X}\rightarrow 0$ such that $\codomdim\Tr_{eA} X>0$ and $ \Tilde{X}\in \cogen \tau_g(DA)$. So, $\Tilde{X}$ is a direct sum of indecomposable modules $M$ with $\codomdim M=0$. Given that $\Ext$ preserves finite direct summands on the first component, we obtain that the $\delta$ is the direct sum of an element in $\Ext_A^1(\Tilde{X}, M)$ with the identity of $M'$, where $M$ is the maximal non-zero direct summand of $\Tr_{eA}X$ with positive dominant dimension and $\Tr_{eA}X=M\oplus M'.$ Since $X$ is indecomposable we must have that $M'=0$ and so $X$ has positive dominant dimension. Hence, $A$ is QF-1.

    We shall now prove (2). Assume that $A$ is QF-1. Then, the torsion pair $(\gen eA, \cogen \tau_g(DA))$ gives that every simple $A$-module is either in $\top (eA)$ or in the socle of $\tau_g(DA)$. Moreover, every simple $A$-module is either in the top of $eA$ or in the socle of $eA$.
\end{proof}

\begin{corollary}\label{cor3dot14}
    Let $A$ be a higher Auslander algebra of global dimension $g$. Assume that there exists a duality functor $(-)^\natural\colon\rmod A\rightarrow \rmod A$ that fixes all the simple $A$-modules and that fixes the faithful projective-injective $A$-module. Then, $A$ is QF-1 if and only if it is semisimple.
\end{corollary}
\begin{proof}
It is clear that every semisimple algebra is QF-1.

Conversely, assume that $A$ is QF-1. Let $eA$ be a faithful projective-injective module.
By assumption, $\top eA=(\top eA)^\natural\cong \soc (eA)^\natural \cong \soc eA$.
  As $\top eA=\soc eA$, we obtain by Proposition \ref{prop4dot5} that $eA$ is a generator, and so $A$ is self-injective. So $A$ must be semisimple.
\end{proof}

By the above corollary, we see that the blocks of representation-finite Schur algebras are not QF-1 algebras. Moreover, there are no (non-semisimple) quasi-hereditary algebras with a simple preserving duality being higher Auslander and QF-1 at the same time.

\begin{remark}\label{remarkdualityontau}
Observe that $\tau_g(D(A))\cong D\transpose D\Omega^{-(g-1)}(A)\cong D\tau_g^{-1}(A)$ both as right and as left $A$-modules. So, $\tau_g^{-1}(A)$ has positive codominant dimension as right $A$-module (resp. as left $A$-module) if and only if $\tau_g(D(A))$ has positive dominant dimension as left $A$-module (resp. as right $A$-module). Since being QF-1 is left-right symmetric, it is enough to study whether (3) implies (1) in Conjecture \ref{mainconjecture}.
\end{remark}

By our results in this section, the Conjecture \ref{higherauscharacterisation} can be shortened to state the following:
\begin{conjecture} \label{conjecturesmaller}
Let $A$ be a higher Auslander algebra of even global dimension $g$.
If \mbox{$\pdim \tau_g(D(A)) \leq g-1$}, then $A$ is QF-1.
\end{conjecture}
The conjecture is true for linear Nakayama algebras with at most 14 simple modules and cyclic Nakayama algebras with at most 12 simple modules, which was verified using a computer.
We remark that the conjecture is true for $g=2$ using a recent result of Zito. As condition (3) of Proposition \ref{prop3dot11} suggests Conjecture \ref{conjecturesmaller} cannot be true for arbitrary higher Auslander algebras. Before we continue, we remark that the conjecture is wrong for odd $g$ instead of even $g$ by the following example:
\begin{example}\label{example3dot15}
The Nakayama algebra with Kupisch series [3,3,3,4] is a higher Auslander algebra and has global dimension $5$ and $\tau_g(D(A))$ has projective dimension 3, but $A$ is not QF-1.
\end{example}

In \cite{Z}, Zito proved the following, which confirms Conjecture \ref{conjecturesmaller} for $g=2$:
\begin{theorem} \label{zitotheorem}
Let $A$ be an Auslander algebra.
Then the following conditions are equivalent:
\begin{enumerate}
\item $A$ is a tilted algebra.
\item $A$ is quasi-tilted.
\item $\pdim \tau \Omega^1(D(A)) \leq 1$.

\end{enumerate}
\end{theorem}

Under Theorem \ref{higherauscharacterisation}, Zito's idea can now be summarised as follows: if $g=2$, then claim 1 of the proof of Proposition \ref{prop3dot11} and Corollary \ref{cor3dot8} gives that $\pdim X\leq 1$ for every $X\in \cogen \tau_2(DA).$ By \citep[Proposition 2.1]{Z} the  torsion pair $(\gen eA, \cogen \tau_2(DA))$ splits. 

In \cite{R}, Ringel attempted to classify the QF-1 algebras of global dimension at most 2. His main result states that such algebras have dominant dimension at least one. Using Theorem \ref{moritatheorem}, such algebras must even have dominant dimension at least 2 and thus are Auslander algebras.
We state this as a corollary of our non-commutative version of the Auslander-Buchsbaum formula:
\begin{corollary}
The following classes of algebras coincide:
\begin{enumerate}
\item QF-1 algebras of global dimension at most two.
\item Tilted Auslander algebras.
\end{enumerate}

\end{corollary}
\begin{proof}
If $A$ is QF-1 of global dimension 2, then $\domdim M+\codomdim M \geq 1$ for every indecomposable $A$-module $M$ and by Morita's theorem \ref{moritatheorem}, $A$ has dominant dimension at least two and thus is an Auslander algebra. By Corollary \ref{cor3dot8}, we know that $$\pdim M + \idim M + \domdim M + \codomdim M=2 \gldim A=4.$$ This implies that $\pdim M + \idim M \leq 3$ and thus $\pdim M <2$ or $\idim M<2$ and thus $A$ is quasi-tilted, which is for Auslander algebras equivalent to $A$ being tilted by Theorem \ref{zitotheorem}.
Now assume that $A$ is a tilted Auslander algebra. Then $A$ is in particular quasi-tilted and thus $\pdim M + \idim M<4$, which implies that $\domdim M + \codomdim M>1$ by using $\pdim M + \idim M + \domdim M + \codomdim M=2 \gldim A=4 $ again.
\end{proof}

Thus both articles \cite{R} and \cite{Z} aim for the classification of the same algebras, namely QF-1 Auslander algebras.
In fact, their classification is quite small as the next result shows. 
\begin{theorem}\label{theoremB}
Let $A$ be a $K$-algebra over an algebraically closed field $K$.
The following are equivalent:
\begin{enumerate}
\item $A$ is a tilted Auslander algebra.
\item $A$ is a QF-1 algebra of global dimension at most 2.
\item $A$ is the Auslander algebra of a path algebra of Dynkin type $A_1, A_2$ or $A_3$ with non-linear orientation.
\item $A$ is the Auslander algebra of a bound quiver algebra $B$ such that every indecomposable $B$-module is projective or injective.
\end{enumerate}
\end{theorem}
\begin{proof}
We already saw that (1) and (2) are equivalent. \newline
Since we work over an algebraically closed field $K$, we can assume that algebras are given by quiver and relations as all finite-dimensional $K$-algebras are Morita equivalent to a bound quiver algebra.
First note that (3) implies (2) and (4).
If $B=KQ$ is a path algebra of Dynkin type $A_1, A_2$ or $A_3$ with non-linear orientation then it is elementary to verify that the Auslander algebra of $B$ is QF-1 as it is of finite representation type and every indecomposable $B$-module is projective or injective. 
Namely, the Auslander algebra of $A_2$ is the Nakayama algebra with Kupisch series [2,2,1] and the Auslander algebra of a Dynkin quiver algebra of non-linear oriented type $A_3$ is isomorphic to the following bound quiver algebra $C=KQ/I$ with relations or its opposite algebra, where $Q$ is the quiver
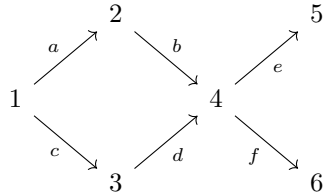
\[\begin{tikzcd}
	& 2 && 5 \\
	1 && 4 \\
	& 3 && 6
	\arrow["b", from=1-2, to=2-3]
	\arrow["a", from=2-1, to=1-2]
	\arrow["c"', from=2-1, to=3-2]
	\arrow["e"', from=2-3, to=1-4]
	\arrow["f"', from=2-3, to=3-4]
	\arrow["d"', from=3-2, to=2-3]
\end{tikzcd}\]
and $I$ is the ideal generated by the relations $ab-cd=0=be=df$.
The algebra $C$ has 18 indecomposable modules and every such module has dominant or codominant dimension at least one. We leave the elementary verification to the reader.

Now we show that (4) implies (3). The condition that every indecomposable module is projective or injective implies that the first syzygy module of every simple module is projective, and thus the algebra has global dimension at most one and thus is hereditary. Furthermore, every simple module is projective or injective and thus the quiver of the algebra consists only of sources or sinks, which implies (3) by also noting that a connected hereditary algebra with $n$ simple modules has more than $2n$ indecomposable modules when $n>3$. \newline
Now assume (2), namely that $A$ is a QF-1 Auslander algebra.
We show (4).
Assume that (4) is not true so that $A$ is Morita equivalent to the Auslander algebra $C=\End_B(M)$ with $M=A \oplus D(A) \oplus X$ with $X$ neither projective nor injective.
Then the indecomposable projective $B$-modules are isomorphic to $P_Y=\Hom_B(M,Y)$ for $Y$ indecomposable summands of $M$ and the indecomposable injective $B$-modules are isomorphic to the $I_Y=D \Hom_B(M,Y)$ for indecomposable summands $Y$ of $M$ and $P_Y$ is injective if and only if $Y$ is injective and dually $I_Y$ is projective if and only if $Y$ is projective.
Then the simple $C$-module $S_T=\top P_T$ for $T$ an indecomposable summand of $X$ has codominant dimension 0 and also dominant dimension zero since it is not isomorphic to the socle of a module of the form $\Hom_B(M,I)$, whose socle is of the form $D(\top(\Hom_B(\nu_B^{-1}(I),M)))$.
\end{proof}

In fact, all QF-1 higher Auslander algebras arise as endomorphism algebras of an $n$-cluster tilting modules whose direct summands are either projective or injective modules. The following result appears in Tachikawa's study of QF-13 rings, see \cite[Chapter 11]{Tac2}. We provide an alternative proof making use of Morita's criterion. 

\begin{theorem}\label{thm3dot19}
    Let $A$ be a finite-dimensional algebra with dominant dimension at least two. Assume that $A$ is QF-1. If $xA$ is not injective for a primitive idempotent $x$ and $Ay$ is not injective for a primitive idempotent $y$, then $\Hom_A(Ax, Ay)\cong xAy=0$. In particular, $A\cong \End_B(X)$ with $\add X=\add B\oplus DB$ for a finite-dimensional algebra $B$.
\end{theorem}
\begin{proof}
    Assume that $xAy\neq 0$. Then, there exists a non-zero element $w\in xAy$. Thus $xw=w=wf$. Let $L$ be a maximal submodule of $Ay$ so that $w\notin L$. Since $y$ is a primitive idempotent, $Ay/L$ has a simple top, namely $\top Ay$. So $Ay/L$ must be indecomposable. By Theorem \ref{thmonedotten}, $Ay/L$ has positive dominant or codominant dimension. The second case cannot happen, since otherwise $\top Ay$ would have positive codominant dimension and then $Ay$ would be a direct summand of a faithful projective-injective module contradicting the assumption on the idempotent $y$. So $Ay/L$ must have positive dominant dimension. Since $w\neq 0$ the assignment $ax\mapsto axw+L$ gives a non-zero map $h\colon Ax\rightarrow Ay/L$. Given $ax\in \rad Ax$, it is clear that $axw\in \rad Ay$. By construction, $\rad Ay\subset L$. So $h$ induces a non-zero map $\top Ax\rightarrow Ay/L$. So, $\top Ax\subset \soc Ay/L$ and $\top Ax$ has positive dominant dimension since $Ay/L$ has. This implies that the injective hull of $\top Ax$, namely $D(xA)$, is projective. But, this contradicts the assumption on the idempotent $x$.
Thus, $xAy=0$. 

Since $\domdim A\geq 2$, $A\cong \End_B(X)$ as $k$-algebras for some generator-cogenerator $X$ over a finite-dimensional algebra $B$. Further, by the Morita-Tachikawa correspondence, $\Hom_B(X, DB)\cong DX$ is a faithful projective-injective left $A$-module and $\Hom_B(B, X)\cong X$ is a faithful projective-injective right $A$-module. Let $M$ be an indecomposable direct summand of $X$ as $B$-module. Observe that $M$ is injective if and only if $\Hom_B(X, M)$ is a direct summand of $\Hom_B(X, DB)$ if and only if $\Hom_B(X, M)$ is injective as left $A$-module. Additionally, $M$ is projective if and only if $\Hom_B(M, X)$ is a direct summand of $\Hom_B(B, X)$ if and only if $\Hom_B(M, X)$ is injective as right $A$-module. Assume now that $M$ is neither projective nor injective. By the first part, we obtain $\Hom_A(\Hom_A(\Hom_B(M, X), A), \Hom_B(X, M))=0$.

    By projectivisation, we have the following isomorphisms of $A$-modules \begin{align*}\Hom_A(\Hom_B(M, X), A)&\cong \Hom_A(\Hom_B(M, X), \Hom_B(X, X))\\&\cong \Hom_A(\Hom_B(DX, DM), \Hom_B(DX, DX))\cong \Hom_B(DM, DX)\\&\cong \Hom_B(X, M). \end{align*}
    Thus, \begin{align*}
        0&=\Hom_A(\Hom_A(\Hom_B(M, X), A), \Hom_B(X, M))\cong \Hom_A(\Hom_B(X, M), \Hom_B(X, M))\\&\cong \Hom_B(M, M).
    \end{align*} This means that $M$ must be the zero module and so the non-zero direct summands of $X$ are projective or injective $B$-modules. Hence the result follows.
\end{proof}

\begin{corollary}\label{cordot13}
    Let $A$ be a higher Auslander QF-1 algebra. Then, $A$ is the endomorphism algebra of a $n$-cluster tilting module having only projective or injective modules as direct summands.
\end{corollary}

The following is \citep[Corollary 11.2]{Tac2} applied to higher Auslander algebras.
\begin{corollary}
    Let $A$ and $A'$ be two higher Auslander QF-1 algebras with global dimension $g$. The algebras $A$ and $A'$ are Morita equivalent if and only if their associated weakly $(g-1)$-representation-finite algebras are Morita equivalent.
\end{corollary}

Tachikawa and Fueller independently establish that for Nakayama algebras, the necessary condition in Theorem \ref{thm3dot19} is also sufficient for QF-1 algebras with positive dominant dimension (see \citep[(11.4)]{Tac2} and \cite{MR232795}).

In fact, a Nakayama higher Auslander algebra $A$ is QF-1 if and only if all the composition factors of the radical of every projective non-injective indecomposable module belong to the socle of $A$. As a consequence, we have the following:

\begin{proposition}
    Let $A$ be a higher Auslander QF-1 Nakayama algebra. Then, $\Ext_A^1(M, N)=0$ for all indecomposable $A$-modules with $\domdim N=0=\codomdim M$. 
\end{proposition}
\begin{proof}
    Let $M$ and $N$ be indecomposable $A$-modules with $\domdim N=0=\codomdim M$. So, $M$ is a quotient of a projective non-injective module and by the discussion above the composition factors of $\Omega^1(M)$ belong to the socle of $A$. Since $N$ has dominant dimension zero its simple socle does not belong to the socle of $A$. Therefore, $\Hom_A(\Omega^1(M), N)=0$. It follows that $\Ext_A^1(M, N)=0$.
\end{proof}

Combining the sufficient condition of Conjecture \ref{conjecturesmaller} with the necessary condition of Theorem \ref{thm3dot19} we obtain the following characterisation of QF-1 higher Auslander algebras.

\begin{theorem}\label{TheoremC}
    Let $A$ be a higher Auslander algebra of global dimension $g$. $A$ is QF-1 if and only the following two conditions are satisfied:
    \begin{enumerate}
        \item $\pdim_A \tau_g(DA)\leq g-1$ or $\idim \underline{A}\leq g-1$ (as right $A$-modules);
        \item If $Ae$ is not injective for a primitive idempotent $e$ and $fA$ is not injective for a primitive idempotent, then $\Hom_A(eA, fA)=fAe=0$.
    \end{enumerate}
\end{theorem}
\begin{proof}
Assume that $A$ is QF-1. By Theorem \ref{thm3dot19}, (2) holds. By Corollary \ref{cor3dot10} and Lemma \ref{formulahighertauinverse}, (1) holds. 

Conversely, assume that $\pdim_A \tau_g(DA)\leq g-1$ and (2) hold. Let $wA$ be a faithful projective-injective right $A$-module. As we have seen before, (1) implies the existence of the torsion pair $(\gen wA, \cogen \tau_g(DA))$ and all the simple $A$-modules with zero codominant dimension are in $\cogen \tau_g(DA)$. Indeed the simple modules with zero codominant dimension are the ones that do not appear in the top of $wA$. Let $X\in \cogen \tau_g(DA)$. Then $\Hom_A(wA, X)=0$ and thus $[X\colon \top wA]=0$. So all composition factors of modules in $\cogen \tau_g(DA)$ are in the socle of $\tau_g(DA)$ and have zero codominant dimension. Now (2) implies that $[I\colon S]=0$ for all injective non-projective indecomposable modules $I$ and simple modules in $\cogen \tau_g(DA)$. Indeed, let $I$ be an injective non-projective indecomposable module and $S$ a simple in the socle of $\tau_g(DA)$. So, $S=\top fA$ for a non-injective $A$-module $fA$ and primitive idempotent $f$ while $I=D(Ae)$ is not projective for some primitive idempotent $e$. So, $Ae$ is not injective. By (2), we have the following implications \begin{align*}
    \Hom_A(eA, fA)=0&\implies [fA\colon \top eA]=0\implies \Hom_A(fA, D(Ae))=\Hom_A(fA, I(\top eA))=0 \\& \implies [I\colon \top fA]=0.
\end{align*}

We now claim that $\Ext_A^1(X, Y)=0$ for every $X\in \cogen \tau_g(DA)$ and indecomposable $A$-module $Y$ with $\domdim_A Y=0$. 
Let $Y$ be an indecomposable $A$-module with zero dominant dimension. Consider the minimal injective presentation
\begin{align}
    0\rightarrow Y\rightarrow I(Y)\rightarrow \Omega^{-1}(Y)\rightarrow 0. \label{eq3dot7}
\end{align} Since $\domdim Y=0$ we obtain that $I(Y)$ is an injective non-projective indecomposable module. So $[I(Y)\colon S]=0$ for all simple modules in $\cogen \tau_g(DA)$. In particular, $[\Omega^{-1}(Y)\colon S]=0$ for all simple modules in $\cogen \tau_g(DA)$. Consequently, $\Hom_A(X, \Omega^{-1}(Y))=0$ for every $X\in \cogen \tau_g(DA)$. By applying $\Hom_A(X, -)$ to  (\ref{eq3dot7}) we obtain that $\Ext_A^1(X, Y)=0$ for every $X\in \cogen \tau_g(DA)$.

Let $W$ be an indecomposable module with zero codominant dimension. Then there exists an exact sequence
\begin{align*}
    \delta\colon 0\rightarrow \Tr_{wA}W\rightarrow W\rightarrow \tilde{W}\rightarrow 0,
\end{align*} with $\tilde{W}\in \cogen \tau_g(DA)$ and $\Tr_{wA}W\in \gen wA$. Assume that $\Tr_{wA} W\cong W_1\oplus W_2$ with $\domdim_A W_1=0$ and $\domdim_A W_2>0$. So $\delta \colon \Ext_A^1(\tilde{W}, W_1\oplus W_2)=\Ext_A^1(\tilde{W}, W_2)$ and this means that $0\rightarrow W_1\rightarrow W_1\rightarrow 0$ is a summand of $\delta$. Since $W$ is indecomposable, $W_1$ must be zero. Then $\domdim_A \Tr_{wA}W\geq 1$. By (1) every module in $\cogen \tau_g(DA)$ has positive dominant dimension, thus $\domdim W\geq 1$. Therefore, $A$ is QF-1 by Theorem \ref{thmonedotten}. Now assume that (2) and $\idim \underline{A}\leq g-1$ hold. By Remark \ref{remarkdualityontau}, $\pdim \tau_g(D(A))\leq g-1$ as right $A^{op}$-module. Observe that if $A^{op}e$ is not injective over $A^{op}$ and $fA^{op}$ is not injective over $A^{op}$, then since $A^{op}e=eA$ we obtain by (2) that $fA^{op}e=eAf=0$. Hence (2) holds for $A^{op}$ and so by the previous part, $A^{op}$ is QF-1. By Theorem  \ref{moritatheorem}, it follows that $A$ is QF-1.
\end{proof}

\section{Examples}\label{Examples}

The second condition in (3) of Proposition \ref{prop3dot11} can occur without $A$ being QF-1. 

\begin{example}
    Let $A$ be the Auslander algebra of $K[X]/(X^2).$  As illustrated by Corollary \ref{cor3dot14}, $A$ cannot be QF-1. Indeed, the top of the projective indecomposable module, $P$, with length 2 has dominant and codominant dimension zero. Moreover, $P$ and its top are the only indecomposable modules with codominant dimension zero. Here, $\soc \ P\cdot AeA$ is the socle of the projective-injective module while $\soc (\top P \cdot AeA)=0$, where $e$ is the idempotent associated with the projective-injective $A$-module.
\end{example}

As Example \ref{example3dot15} shows, the first condition in (3) of  Proposition \ref{prop3dot11} can occur without $A$ being QF-1. There, it is the second condition that fails.

Contrary to tilted Auslander algebras, QF-1 higher Auslander algebras are more common. In fact, the following example illustrates that there are QF-1 higher Auslander algebras of global dimension $g$ for every dimension $g\geq 2$.

\begin{example} \label{example4dot1}
    Let $K$ be an algebraically closed field and $n\geq 2$ a natural number. We define $A_n$ to be the bound quiver $K$-algebra
$\begin{tikzcd}
	1 \arrow[r, "\alpha_1"] & 2 \arrow[r, "\alpha_2"] & \cdots \arrow[r] & n
\end{tikzcd}$ with relations $0=\alpha_i\alpha_{i+1},$ $i=1, \ldots, n-1$. We write $S(i)$ to denote the simple module associated with the vertex $i$ and $P(i)$ the projective cover of $S(i)$ for every $i=1, \ldots, n$.
The module $P(1)\oplus P(2)\oplus \cdots \oplus P(n-1)$ is a faithful projective-injective module and the minimal injective resolution of $S(n)=P(n)$ coincides with the minimal projective resolution of $S(1)$. Hence, $\domdim A_n=n-1$. Moreover, $\pdim_A S(i)=n-i$ for every $i=1, \ldots, n$ and thus $\gldim A_n=n-1$ which means that $A_n$ is a higher Auslander algebra.

Observe that $\tau_{n-1}DA_n\cong\tau_{n-1}(S(1))\cong\tau(S(n-1)) \cong S(n)=P(n)$ is projective, so the torsion pair induced by the faithful projective-injective $(\gen (P(1)\oplus \cdots P(n-1)), \add P(n))$ splits. So, the algebra $A_n$ is QF-1. 
We can observe that $B:=\End(P(1)\oplus \cdots \oplus P(n-1))$ is the algebra $A_{n-1}$ (e.g it is semisimple if $n=2$) and $DP\cong \Hom_A(P, DA)$ is the right $B$-module with the same additive closure as $DB\oplus B$. That is, $A_{n-1}$ is $(n-2)$-representation-finite and by induction $A_n$ is $(n-1)$-representation-finite with $A_n\oplus DA_n$ as $(n-1)$-cluster tilting module.

By \citep[Theorem 1.19]{Iya2}, the algebras $A_n$ are Morita equivalent to Iyama's algebras $T_m^{(n-1)}(K)$ for some $m$. But since $A_2$ is Morita equivalent to the triangular matrix algebra $T_2(K)=T_2^{1}(K)$ it follows that $A_n$ is Morita equivalent to $T_2^{(n-1)}(K)$.
\end{example}

Actually,  if a QF-1 higher Auslander algebra is $g$-representation-finite, then the generator-cogenerator $A\oplus DA$ is the unique $g$-cluster tilting module.

\begin{proposition} \label{qf1nrephigher}
    Let $A$ be a higher Auslander algebra of global dimension $g$ and $P$ a faithful projective-injective $A$-module. If $A$ is QF-1 and $g$-representation-finite, then $A\oplus DA$ is a $g$-cluster tilting module and the algebra $\End_A(P)$ is $(g-1)$-representation-finite.
\end{proposition}
\begin{proof}
    By assumption, $\pdim_A \tau_g(DA)<g.$ Hence, since $\tau_g\cong D\Ext_A^g(-, A)$ we get that $\tau_g^2(DA)=0.$ So, $\oplus_{n\geq 0} \tau_g^n(DA)=DA\oplus \tau_g(DA).$ Since $A$ is $g$-representation-finite, it follows by \citep[Theorem 1.6]{Iya2} that $DA\oplus \tau_g(DA)$ is the unique $g$-cluster tilting $A$-module. So, all non-injective projective indecomposable occur as direct summands of $\tau_g(DA)$. The number of non-isomorphic direct summands of $\tau_g(DA)$ coincides with the number of non-isomorphic injective  indecomposable modules with projective dimension $g$ (see \citep[Proposition 1.12]{Iya2}). By Corollary \ref{cor3dot8}, these are precisely the injective non-projective indecomposable modules. Thus, $\tau_g(DA)$ must be the direct sum of the projective non-injective indecomposable modules. Thus, $A\oplus DA$ is a $g$-cluster tilting module and $\Ext_A^i(DA, A)=0$ for $i=1, \ldots, g-1.$ By \citep[Theorem 1.20]{Iya2}, $\End_A(P)$ is $(g-1)$-representation-finite.
\end{proof}

\begin{corollary}
Let $A$ be a higher Auslander algebra of global dimension $g$ that is also $g$-representation-finite. Then $A$ is isomorphic to $T_n^{(m)}$ for $n=2$.
\end{corollary}
\begin{proof}
We already saw in the previous example that $T_2^{(m)}$, which are the Nakayama algebras with Kupisch series [2,2,....,2,1], have this property.
The corollary follows now by Proposition \ref{qf1nrephigher} and the fact that for $n>2$ the unique $g$-cluster tilting module is not of the form $A \oplus D(A)$ for $A=T_n^{(m)}$, see the construction of those cluster-tilting modules in Section 3 of \cite{Iya2}.  
\end{proof}
The next example illustrates that the associated algebra $eAe$ of a QF-1 higher Auslander algebra $A$ of global dimension $g$ is not necessarily $(g-1)$-representation-finite.

\begin{example}
    Let $A$ be the Nakayama algebra with Kupisch series $[2, 3, 3, 3, 3, 2, 2, 1]$. 
    The Auslander-Reiten quiver is the following

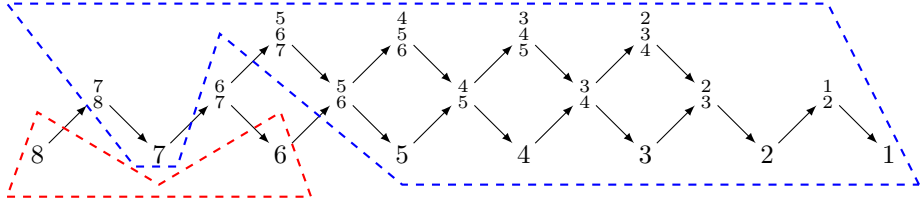
\[
\begin{tikzpicture}[scale=0.8,
fl/.style={->,shorten <=6pt, shorten >=6pt,>=latex}]
\coordinate (13) at (0,0) ;
\coordinate (14) at (1,1) ;
\coordinate (15) at (2,2) ;
\coordinate (24) at (2,0) ;
\coordinate (25) at (3,1) ;
\coordinate (26) at (4,2) ;
\coordinate (35) at (4,0) ;
\coordinate (36) at (5,1) ;
\coordinate (37) at (6,2) ;
\coordinate (46) at (6,0) ;
\coordinate (47) at (7,1) ;
\coordinate (48) at (8,2) ;
\coordinate (57) at (8,0) ;
\coordinate (58) at (9,1) ;
\coordinate (59) at (10,2) ;
\coordinate (68) at (10,0) ;
\coordinate (69) at (11,1) ;
\coordinate (610) at (12,2) ;
\coordinate (79) at (12,0) ;
\coordinate (710) at (13,1) ;
\coordinate (711) at (14,2) ;
\coordinate (810) at (14,0) ;


\draw[fl] (13) -- (14) ;
\draw[fl] (14) -- (24) ;
\draw[fl] (24) --(25) ;
\draw[fl] (25) --(35) ;
\draw[fl] (25) --(26) ;
\draw[fl] (35) --(36) ;
\draw[fl] (36) --(37) ;
\draw[fl] (26) --(36) ;
\draw[fl] (46) --(47) ;
\draw[fl] (47) --(48) ;
\draw[fl] (37) --(47) ;
\draw[fl] (36) --(46) ;
\draw[fl] (57) --(58) ;
\draw[fl] (58) --(59) ;
\draw[fl] (48) --(58) ;
\draw[fl] (47) --(57) ;
\draw[fl] (58) --(68) ;
\draw[fl] (68) --(69) ;
\draw[fl] (59) --(69) ;
\draw[fl] (79) --(710) ;
\draw[fl] (69) --(79) ;
\draw[fl] (710) --(810) ;
\draw (13) node[scale=1] {$8$} ;
\draw (14) node[scale=1] {$\begin{smallmatrix}
     7 \\ 8 
 \end{smallmatrix}$} ;
\draw (24) node[scale=1] {$7$} ;
\draw (25) node[scale=1] {$\begin{smallmatrix}
     6 \\ 7
 \end{smallmatrix}$} ;
\draw (26) node[scale=1] {$\begin{smallmatrix}
     5\\ 6\\ 7
 \end{smallmatrix}$} ;
\draw (35) node[scale=1] {$6$} ;
\draw (36) node[scale=1] {$\begin{smallmatrix}
     5 \\ 6
 \end{smallmatrix}$} ;
\draw (37) node[scale=1] {$\begin{smallmatrix}
     4 \\ 5 \\ 6
 \end{smallmatrix}$} ;
\draw (46) node[scale=1] {$5$} ;
\draw (47) node[scale=1] {$\begin{smallmatrix}
     4 \\ 5
 \end{smallmatrix}$} ;
\draw (48) node[scale=1] {$\begin{smallmatrix}
     3\\ 4\\ 5
 \end{smallmatrix}$} ;
\draw (57) node[scale=1] {$4$} ;
\draw (58) node[scale=1] {$\begin{smallmatrix}
     3 \\ 4
 \end{smallmatrix}$} ;
\draw (59) node[scale=1] {$\begin{smallmatrix}
     2 \\ 3 \\ 4
 \end{smallmatrix}$} ;
\draw (68) node[scale=1] {$3$} ;
\draw (69) node[scale=1] {$\begin{smallmatrix}
     2 \\ 3
 \end{smallmatrix}$} ;
\draw (79) node[scale=1] {$2$} ;
\draw (710) node[scale=1] {$\begin{smallmatrix}
      1 \\ 2
 \end{smallmatrix}$} ;
\draw (810) node[scale=1] {$1$} ;
\draw[thick, dashed, red] (-0.5,-0.7) --   (4.5, -0.7) -- (4, 0.7)  -- (2, -0.5) -- (0, 0.7) --  cycle ;
\draw[thick, dashed, blue]  (14.5, -0.5) -- (13, 2.5) -- (-0.5, 2.5) -- (1.6, -0.2) -- (2.3, -0.2) -- (3, 2) -- (6, -0.5) -- cycle;
\end{tikzpicture}
    \]

    $A$ has global dimension $4$ and dominant dimension $4$. Here, the direct sum of the projective covers of the simple modules indexed by $1, 2, 3, 4, 5$ and $7$, $P$, is the faithful projective-injective $A$-module.
    So, $B=\End_A(P)$ is the Nakayama algebra with Kupisch series $[2, 3, 3, 2, 2, 1]$ and $\add \Hom_A(P, DA)=\add B\oplus DB$.
    From the Auslander-Reiten quiver, it is clear that $A$ is QF-1. Under the dashed blue, it is marked the subcategory of indecomposable modules with positive codominant dimension while in the red it is marked the submodules of $\tau_g(DA)$. 
    Since $A$ is QF-1, $\tau_g^2(DA)=0$. On the other hand $DA\oplus \tau_g(DA)$ is not a generator, and thus $DA\oplus \tau_g(DA)$ cannot be cluster tilting. This means that $A$ is not $4$-representation-finite. Observe now that $\Ext_A^3(DA, A)\neq 0$ since $\Ext_A^3(\begin{smallmatrix}
        2 \\3
    \end{smallmatrix}, \begin{smallmatrix}
        6 \\ 7
    \end{smallmatrix})\cong \Ext_A^1(\begin{smallmatrix}
        5 \\6 
    \end{smallmatrix}, \begin{smallmatrix}
        6 \\ 7
    \end{smallmatrix})\neq 0$. By \citep[Theorem 1.20]{Iya2}, $\End_A(P)$ is not $3$-representation-finite.
\end{example}

Next we give an example that illustrates that there is no version of Lemma \ref{formulahighertauinverse} for $\tau_g(DA)$ and $\overline{A}$.
\begin{example}
Let $A$ be the Nakayama algebra with Kupisch series $[5, 5, 5, 5, 5, 7, 6]$. Then, $A$ is a higher Auslander algebra of global dimension $5$, the unique simple direct summand of $\tau_5(DA)$ is isomorphic to the simple module with top $7$ while the unique simple direct summand of $\overline{A}$ is isomorphic to simple modules with top $1$. 
\end{example}

The next example gives a non-trivial example of a QF-1 higher Auslander algebra that is not of finite representation type.
\begin{example} \label{nontrivialexampleQF1highaus}
Let $K$ be the finite field with three elements and consider the bound quiver algebra $A=KQ/I$ with quiver $Q$ given by
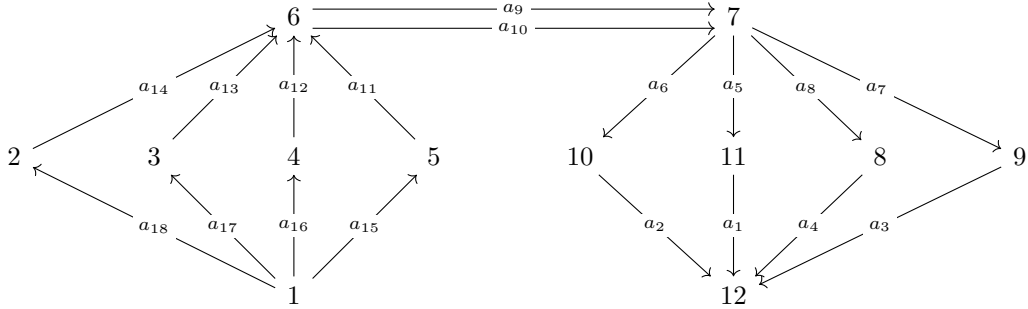
\[\begin{tikzcd}[column sep=9ex,row sep=9ex]
	&& 6 &&& 7 \\
	2 & 3 & 4 & 5 & 10 & 11 & 8 & 9 \\
	&& 1 &&& 12
	\arrow["{a_9}"{description}, shift left=1.5, from=1-3, to=1-6]
	\arrow["{a_{10}}"{description}, shift right=1.5, from=1-3, to=1-6]
	\arrow["{a_6}"{description}, from=1-6, to=2-5]
	\arrow["{a_5}"{description}, from=1-6, to=2-6]
	\arrow["{a_8}"{description}, from=1-6, to=2-7]
	\arrow["{a_7}"{description}, from=1-6, to=2-8]
	\arrow["{a_{14}}"{description}, from=2-1, to=1-3]
	\arrow["{a_{13}}"{description}, from=2-2, to=1-3]
	\arrow["{a_{12}}"{description}, from=2-3, to=1-3]
	\arrow["{a_{11}}"{description}, from=2-4, to=1-3]
	\arrow["{a_2}"{description}, from=2-5, to=3-6]
	\arrow["{a_1}"{description}, from=2-6, to=3-6]
	\arrow["{a_4}"{description}, from=2-7, to=3-6]
	\arrow["{a_3}"{description}, from=2-8, to=3-6]
	\arrow["{a_{18}}"{description}, from=3-3, to=2-1]
	\arrow["{a_{17}}"{description}, from=3-3, to=2-2]
	\arrow["{a_{16}}"{description}, from=3-3, to=2-3]
	\arrow["{a_{15}}"{description}, from=3-3, to=2-4]
\end{tikzcd}\]
and relations $$I=\left\langle \begin{alignedat}{2} a_{18}a_{14}+a_{17} a_{13}-2 a_{16} a_{12}, &a_{17}a_{13}-a_{16} a_{12}+a_{15} a_{11}, a_{14} a_9-a_{14}a_{10}, a_{13}a_9+ a_{13} a_{10},\\ a_{12} a_9, a_{11} a_{10}, &a_9 a_6, 
  a_9 a_8- a_{10} a_8, a_9 a_7+a_{10} a_7, a_{10} a_5, \\ a_8 a_4-a_7 &a_3-2a_6 a_2, a_7 a_3+a_6 a_2-a_5a_1 \end{alignedat} \right\rangle.$$
Clearly, $A$ is of infinite representation type as it contains a Kroenecker subquiver at the vertices 6 to 7 and $A$ is a higher Auslander algebra of global dimension 3.
Thus directly using Morita's theorem \ref{moritatheorem} is not possible in a finite time to deduce that $A$ is QF-1. But our Theorem \ref{TheoremC} requires only to check two conditions and these can be done in finite time although it is quite tedious for this example. By Theorem  \ref{TheoremC}, $A$ is QF-1. In the Appendix \ref{appendix}, we illustrate how can this fact be verified using the computer algebra system \cite{QPA}.
\end{example}

\appendix

\section{Testing whether a higher Auslander algebra is QF-1 using QPA}
\label{appendix}

This appendix is intended for the readers who are familiar with the GAP package \cite{QPA} and would like to verify whether a given bound quiver algebra is a QF-1 higher Auslander algebra. To test whether a given higher Auslander algebra is QF-1 we use the following GAP program:

\begin{tiny}
\begin{verbatim}
DeclareOperation("QF1testforhigherAuslanderalgebras",[IsList]);

InstallMethod(QF1testforhigherAuslanderalgebras, "for higher Auslander algebras", [IsList],0,function(LIST)

local A,CoRegA,g,B,projA,projB,U1,U2,W,i,T,UL1,UUL1,AmAfA,RegA;

A:=LIST[1];
g:=GlobalDimensionOfAlgebra(A,33);
B:=OppositeAlgebra(A);
projB:=IndecProjectiveModules(B);
UL1:=Filtered(projB,x->IsInjectiveModule(x)=true);
UUL1:=DirectSumOfQPAModules(UL1);
Af:=StarOfModule(UUL1);
RegA:=DirectSumOfQPAModules(IndecProjectiveModules(A));
AmAfA:=CoKernel(TraceOfModule(Af,RegA));
projA:=IndecProjectiveModules(A);
U1:=Filtered(projA,x->IsInjectiveModule(x)=false);
projB:=IndecProjectiveModules(B);
U2:=Filtered(projB,x->IsInjectiveModule(x)=false);
W:=[];for i in U2 do for j in U1 do Append(W,[Size(HomOverAlgebra(StarOfModule(i),j))]);od;od;
return(InjDimensionOfModule(AmAfA,g)<=g-1 and Maximum(W)=0);

end);
\end{verbatim}
\end{tiny}

This program can be entered directly into GAP terminal after loading the QPA package (in case it is not loaded automatically). The program returns true if a given higher Auslander algebra (given by quiver and relations) is QF-1 and returns false otherwise. 

Before using the program the reader is advised to check first whether a given algebra $A$ is a higher Auslander algebra by evaluating whether the following two functions return the same value for large enough $n\in \mathbb{N}:$

\begin{tiny}
\begin{verbatim}
    g:=GlobalDimensionOfAlgebra(A,n);
    d:=DominantDimensionOfAlgebra(A,n);
\end{verbatim}
\end{tiny}

We can enter the algebra in Example \ref{nontrivialexampleQF1highaus} in QPA and test if is a higher Auslander algebra as follows:
\begin{tiny}
\begin{verbatim}
Q:=Quiver( ["v1","v2","v3","v4","v5","v6","v7","v8","v9","v10","v11","v12"], [["v1","v2","a18"],
["v1","v3","a17"],["v1","v4","a16"],["v1","v5","a15"],["v2","v6","a14"],
["v3","v6","a13"],["v4","v6","a12"],["v5","v6","a11"],["v6","v7","a9"],["v6","v7","a10"],
["v7","v8","a8"],["v7","v9","a7"],["v7","v10","a6"],["v7","v11","a5"],["v8",
"v12","a4"],["v9","v12","a3"],["v10","v12","a2"],["v11","v12","a1"]] );KQ:=PathAlgebra(GF(3),Q);AssignGeneratorVariables(KQ);rel:=
[ (1)*a18*a14+(1)*a17*a13+(-2)*a16*a12, (1)*a17*a13+(-1)*a16*a12+(1)*a15*a11, 
(1)*a14*a9+(-1)*a14*a10, (1)*a13*a9+(1)*a13*a10, (1)*a12*a9, 
(1)*a11*a10, (1)*a9*a6, 
  (1)*a9*a8+(-1)*a10*a8, (1)*a9*a7+(1)*a10*a7, (1)*a10*a5, 
  (1)*a8*a4+(-1)*a7*a3+(-2)*a6*a2, (1)*a7*a3+(1)*a6*a2+(-1)*a5*a1 ];A:=KQ/rel;
g:=GlobalDimensionOfAlgebra(A,33);
d:=DominantDimensionOfAlgebra(A,33);g=d;
\end{verbatim}
\end{tiny}
We indeed see that this algebra has global and dominant dimension equal to three and then we use the program to see that it is indeed QF1 as follows:
\begin{tiny}
\begin{verbatim}
QF1testforhigherAuslanderalgebras([A]);
\end{verbatim}
\end{tiny}

\section*{Acknowledgement}
The second author is thankful to Osamu Iyama for useful discussions regarding Proposition \ref{taustablehomformula}. The proof of Proposition \ref{inequpdim} is due to {\O}yvind Solberg and we are thankful to him for allowing us to use his proof here. We profited from the use of the GAP-package \cite{QPA}.
\bibliographystyle{alphaurl}
\bibliography{bibarticle}

\begin{thebibliography}{QPA22}

\bibitem[AK96]{zbMATH00941221}
Ibrahim Assem and Otto Kerner.
\newblock Constructing torsion pairs.
\newblock {\em J. Algebra}, 185(1):19--41, 1996.
\newblock \href {https://doi.org/10.1006/jabr.1996.0310}
  {\path{doi:10.1006/jabr.1996.0310}}.

\bibitem[APT92]{APT}
Maurice Auslander, Mar\'ia~In\'es Platzeck, and Gordana Todorov.
\newblock Homological theory of idempotent ideals.
\newblock {\em Trans. Am. Math. Soc.}, 332(2):667--692, 1992.
\newblock \href {https://doi.org/10.2307/2154190} {\path{doi:10.2307/2154190}}.

\bibitem[ARS97]{ARS}
Maurice Auslander, Idun Reiten, and Sverre~O. Smal\o.
\newblock {\em Representation theory of {A}rtin algebras}, volume~36 of {\em
  Cambridge Studies in Advanced Mathematics}.
\newblock Cambridge University Press, Cambridge, 1997.
\newblock Corrected reprint of the 1995 original.

\bibitem[ASS06]{ASS}
Ibrahim Assem, Daniel Simson, and Andrzej Skowro\'nski.
\newblock {\em Elements of the representation theory of associative algebras.
  {V}ol. 1}, volume~65 of {\em London Mathematical Society Student Texts}.
\newblock Cambridge University Press, Cambridge, 2006.
\newblock \href {https://doi.org/10.1017/CBO9780511614309}
  {\path{doi:10.1017/CBO9780511614309}}.

\bibitem[CIM24]{CIM}
Xiao-Wu Chen, Srikanth~B. Iyengar, and Ren{\'e} Marczinzik.
\newblock Homological dimensions of the {Jacobson} radical.
\newblock {\em Proc. Am. Math. Soc., Ser. B}, 11:211--223, 2024.
\newblock \href {https://doi.org/10.1090/bproc/187}
  {\path{doi:10.1090/bproc/187}}.

\bibitem[CP24]{CPsa}
Tiago Cruz and Chrysostomos Psaroudakis.
\newblock Relative {Auslander--Gorenstein} pairs, 2024.
\newblock \href {http://arxiv.org/abs/2302.10704} {\path{arXiv:2302.10704}}.

\bibitem[Ful68]{MR232795}
Kent~R. Fuller.
\newblock Generalized uniserial rings and their {K}upisch series.
\newblock {\em Math. Z.}, 106:248--260, 1968.
\newblock \href {https://doi.org/10.1007/BF01110273}
  {\path{doi:10.1007/BF01110273}}.

\bibitem[GLS08]{GLS}
Christof Geiss, Bernard Leclerc, and Jan Schr{\"o}er.
\newblock Preprojective algebras and cluster algebras.
\newblock In {\em Trends in representation theory of algebras and related
  topics. Proceedings of the 12th international conference on representations
  of algebras and workshop (ICRA XII), Toru\'n, Poland, August 15--24, 2007.},
  pages 253--283. Z{\"u}rich: European Mathematical Society (EMS), 2008.

\bibitem[HRS96]{HRS}
Dieter Happel, Idun Reiten, and Sverre~O. Smal{\o}.
\newblock {\em Tilting in abelian categories and quasitilted algebras}, volume
  575 of {\em Mem. Am. Math. Soc.}
\newblock Providence, RI: American Mathematical Society (AMS), 1996.
\newblock \href {https://doi.org/10.1090/memo/0575}
  {\path{doi:10.1090/memo/0575}}.

\bibitem[IR08]{IR}
Osamu Iyama and Idun Reiten.
\newblock Fomin-{Zelevinsky} mutation and tilting modules over {Calabi}-{Yau}
  algebras.
\newblock {\em Am. J. Math.}, 130(4):1087--1149, 2008.
\newblock \href {https://doi.org/10.1353/ajm.0.0011}
  {\path{doi:10.1353/ajm.0.0011}}.

\bibitem[Iya07]{Iya}
Osamu Iyama.
\newblock Higher-dimensional {A}uslander-{R}eiten theory on maximal orthogonal
  subcategories.
\newblock {\em Adv. Math.}, 210(1):22--50, 2007.
\newblock \href {https://doi.org/10.1016/j.aim.2006.06.002}
  {\path{doi:10.1016/j.aim.2006.06.002}}.

\bibitem[Iya08]{Iya3}
Osamu Iyama.
\newblock Auslander-{Reiten} theory revisited.
\newblock In {\em Trends in representation theory of algebras and related
  topics. Proceedings of the 12th international conference on representations
  of algebras and workshop (ICRA XII), Toru\'n, Poland, August 15--24, 2007.},
  pages 349--397. Z{\"u}rich: European Mathematical Society (EMS), 2008.

\bibitem[Iya11]{Iya2}
Osamu Iyama.
\newblock Cluster tilting for higher {Auslander} algebras.
\newblock {\em Adv. Math.}, 226(1):1--61, 2011.
\newblock \href {https://doi.org/10.1016/j.aim.2010.03.004}
  {\path{doi:10.1016/j.aim.2010.03.004}}.

\bibitem[J{\o}r98]{MR1644217}
Peter J{\o}rgensen.
\newblock Non-commutative graded homological identities.
\newblock {\em J. London Math. Soc. (2)}, 57(2):336--350, 1998.
\newblock \href {https://doi.org/10.1112/S0024610798006164}
  {\path{doi:10.1112/S0024610798006164}}.

\bibitem[Li23]{S}
Shen Li.
\newblock Higher {Auslander} algebras of finite representation type, 2023.
\newblock \href {http://arxiv.org/abs/2308.10433} {\path{arXiv:2308.10433}}.

\bibitem[Miy00]{Mi}
Jun-ichi Miyachi.
\newblock Injective resolutions of {Noetherian} rings and cogenerators.
\newblock {\em Proc. Am. Math. Soc.}, 128(8):2233--2242, 2000.
\newblock \href {https://doi.org/10.1090/S0002-9939-00-05305-3}
  {\path{doi:10.1090/S0002-9939-00-05305-3}}.

\bibitem[Mor58]{zbMATH03171195}
Kiiti Morita.
\newblock On algebras for which every faithful representation is its own second
  commutator.
\newblock {\em Math. Z.}, 69:429--434, 1958.
\newblock \href {https://doi.org/10.1007/BF01187420}
  {\path{doi:10.1007/BF01187420}}.

\bibitem[Nis88]{MR968204}
Kenji Nishida.
\newblock A characterization of {G}orenstein orders.
\newblock {\em Tsukuba J. Math.}, 12(2):459--468, 1988.
\newblock \href {https://doi.org/10.21099/tkbjm/1496160842}
  {\path{doi:10.21099/tkbjm/1496160842}}.

\bibitem[QPA22]{QPA}
The QPA-team.
\newblock {\em QPA -- Quivers, path algebras and representations -- a GAP
  package, Version 1.33}, 2022.
\newblock URL: \url{https://folk.ntnu.no/oyvinso/QPA/}.

\bibitem[Rin73]{R}
Claus~Michael Ringel.
\newblock {${\rm QF}-1$} rings of global dimension {$\leq 2$}.
\newblock {\em Canadian J. Math.}, 25:345--352, 1973.
\newblock \href {https://doi.org/10.4153/CJM-1973-033-7}
  {\path{doi:10.4153/CJM-1973-033-7}}.

\bibitem[RT75]{RT}
Claus~Michael Ringel and Hiroyuki Tachikawa.
\newblock {QF}-3 rings.
\newblock {\em J. Reine Angew. Math.}, 272:49--72, 1975.
\newblock URL: \url{https://eudml.org/doc/151523}.

\bibitem[Tac73]{Tac2}
Hiroyuki Tachikawa.
\newblock {\em Quasi-{F}robenius rings and generalizations. {${\rm QF}-3$} and
  {${\rm QF}-1$} rings}, volume Vol. 351 of {\em Lecture Notes in Mathematics}.
\newblock Springer-Verlag, Berlin-New York, 1973.
\newblock Notes by Claus Michael Ringel.

\bibitem[Thr48]{T}
Robert~McDowell Thrall.
\newblock Some generalizations of quasi-{Frobenius} algebras.
\newblock {\em Trans. Am. Math. Soc.}, 64:173--183, 1948.
\newblock \href {https://doi.org/10.2307/1990561} {\path{doi:10.2307/1990561}}.

\bibitem[WZ00]{MR1741569}
Q.~S. Wu and J.~J. Zhang.
\newblock Some homological invariants of local {PI} algebras.
\newblock {\em J. Algebra}, 225(2):904--935, 2000.
\newblock \href {https://doi.org/10.1006/jabr.1999.8179}
  {\path{doi:10.1006/jabr.1999.8179}}.

\bibitem[WZ01]{MR1848957}
Q.-S. Wu and J.~J. Zhang.
\newblock Homological identities for noncommutative rings.
\newblock {\em J. Algebra}, 242(2):516--535, 2001.
\newblock \href {https://doi.org/10.1006/jabr.2001.8817}
  {\path{doi:10.1006/jabr.2001.8817}}.

\bibitem[Yam96]{Yam}
Kunio Yamagata.
\newblock Frobenius algebras.
\newblock In {\em Handbook of algebra, {V}ol. 1}, volume~1 of {\em Handb.
  Algebr.}, pages 841--887. Elsevier/North-Holland, Amsterdam, 1996.
\newblock \href {https://doi.org/10.1016/S1570-7954(96)80028-3}
  {\path{doi:10.1016/S1570-7954(96)80028-3}}.

\bibitem[Zit20]{Z}
Stephen Zito.
\newblock Auslander algebras which are tilted.
\newblock {\em Arch. Math. (Basel)}, 115(3):241--246, 2020.
\newblock \href {https://doi.org/10.1007/s00013-020-01471-2}
  {\path{doi:10.1007/s00013-020-01471-2}}.

\end{thebibliography}

\end{document}